\documentclass[12pt]{amsart}

\usepackage{enumerate}

\usepackage[foot]{amsaddr}
\usepackage[english]{babel}

\allowdisplaybreaks

\usepackage[
pdftex,hyperfootnotes]{hyperref}
\hypersetup{
	colorlinks=true,
	linkcolor=NavyBlue, 
	urlcolor=RoyalPurple,
	citecolor=OliveGreen,
	pdftitle={Bidiagonal factorization of Hessenberg or Banded matrices},
	bookmarks=true,
}
\usepackage[usenames,dvipsnames,svgnames,table,x11names]{xcolor}
\usepackage{pgfplots}
\usepackage{tikz}
\usepackage{tikz-3dplot}
\usetikzlibrary{automata,quotes, chains,matrix,calc,shadows,shapes.callouts,shapes.geometric,shapes.misc,positioning,patterns,decorations.shapes,
	decorations.pathmorphing,decorations.markings,decorations.fractals,decorations.pathreplacing,shadings,fadings,arrows.meta,bending}

\usepackage{nicematrix,drawmatrix}
\usepackage[utf8]{inputenc}

\usepackage{comment}

\usepackage[textwidth=17.5cm,textheight=22.275cm,
height=
24.275cm,
width=18cm]{geometry}

\usepackage{amssymb,latexsym,amsmath,amsthm,bm}
\usepackage{mathrsfs}
\usepackage{mathtools,arydshln,mathdots}
\mathtoolsset{showonlyrefs}
\usepackage{tcolorbox}

\usepackage[table]{xcolor}

\usepackage{Baskervaldx}
\usepackage[]{newtxmath}

\usepackage[all]{xy}
\hypersetup{
	colorlinks=true,
	linkcolor=blue}
\date{}

\theoremstyle{plain}

\newtheorem{Theorem}{Theorem}[section]
\newtheorem{Corollary}[Theorem]{Corollary}

\newtheorem{Proposition}[Theorem]{Proposition}
\newtheorem{Definition}[Theorem]{Definition}
\newtheorem{Remark}[Theorem]{Remark}

\renewcommand{\d}{\operatorname{d}}

\newcommand{\sgn}{\operatorname{sgn}}

\newcommand{\diag}{\operatorname{diag}}

\newcommand{\C}{\mathbb{C}}

\newcommand{\N}{\mathbb{N}}

\begin{document}
	
		\title[Christoffel Perturbations for Mixed Multiple Orthogonality]{General Christoffel Perturbations for  \\Mixed Multiple Orthogonal Polynomials}

	\author[M Mañas]{Manuel Mañas$^{1}$}
	
	\author[M Rojas]{Miguel Rojas$^{2}$}
	\address{Departamento de Física Teórica, Universidad Complutense de Madrid, Plaza Ciencias 1, 28040-Madrid, Spain}
	\email{$^{1}$manuel.manas@ucm.es}
	\email{$^{2}$migroj01@ucm.es}


		\begin{abstract}
Performing both right and left multiplication operations using general regular matrix polynomials, which need not be monic and may possess leading coefficients of arbitrary rank, on a rectangular matrix of measures associated with mixed multiple orthogonal polynomials, reveals corresponding Christoffel formulas. These formulas express the perturbed mixed multiple orthogonal polynomials in relation to the original ones. Utilizing the divisibility theorem for matrix polynomials, we establish a criterion for the existence of perturbed orthogonality, expressed through the non-cancellation of certain $\tau$ determinants.
	\end{abstract}
	
	\maketitle
\tableofcontents
	\section{Introduction}
	
Multiple orthogonal polynomials constitute a versatile class of polynomials with broad applications spanning diverse fields in mathematics and engineering. Unlike their orthogonal counterparts, which are associated with a single weight function, multiple orthogonal polynomials are linked to several weight functions and measures simultaneously. These polynomials serve as indispensable tools in numerical analysis, approximation theory, and mathematical physics, offering robust solutions to complex problems characterized by simultaneous orthogonality conditions.

Traditionally, multiple orthogonal polynomials have been closely connected with the theory of Hermite–Padé and its applications in constructive function theory. For insightful introductions to multiple orthogonal polynomials, one can consult Nikishin and Sorokin's book \cite{nikishin_sorokin} and the chapter by Van Assche in \cite[Ch. 23]{Ismail}. Additionally, their relation with integrable systems is elaborated upon in \cite{afm}, with a basic yet inspiring introduction provided in \cite{andrei_walter}. Studies on the asymptotic behavior of zeros can be found in \cite{Aptekarev_Kaliaguine_Lopez}, while a Gauss–Borel perspective is explored in \cite{afm} and applications to random matrix theory are detailed in \cite{Bleher_Kuijlaars}. 

Mixed multiple orthogonal polynomials, along with the corresponding Riemann–Hilbert problem, have found applications in various areas such as Brownian bridges or non-intersecting Brownian motions \cite{Evi_Arno}, as well as in the study of multicomponent Toda systems, cf. \cite{adler,afm}. Moreover, mixed multiple orthogonal polynomials have been utilized in number theory, notably in the proof by Apéry \cite{Apery} that $\zeta(3)$ is irrational, and in demonstrating the irrationality of certain values of the $\zeta$ function at odd integers \cite{Ball_Rivoal}. 

In \cite{ulises}, the authors delved into the logarithmic and ratio asymptotics of linear forms constructed from a Nikishin system. This system satisfies orthogonality conditions with respect to a set of measures generated by a second Nikishin system.  Furthermore, in \cite{ulises2}, a comprehensive investigation was conducted into a broad class of mixed-type multiple orthogonal polynomials, along with an examination of the properties of their corresponding zeros.

Recent research has underscored the importance of mixed multiple orthogonal polynomials in the Favard spectral description of banded bounded semi-infinite matrices. This connection has been explored in various works such as \cite{aim, phys-scrip, BTP, Contemporary}, with further insights available in \cite{laa}. Additionally, these polynomials play a crucial role in the realm of Markov chains and random walks extending beyond birth and death processes, as demonstrated in \cite{CRM,finite,hypergeometric,JP}.

In 1858, the German mathematician Elwin Christoffel \cite{christoffel} embarked on an exploration of Gaussian quadrature rules, with the goal of revealing explicit formulas linking sequences of orthogonal polynomials under different measures. Specifically, he investigated the Lebesgue measure $\mathrm{d}\mu = \mathrm{d}x$ and a modified measure $\d\hat{\mu}(x) = p(x) \d\mu(x)$, where $p(x) = (x - q_1) \cdots (x - q_N)$ is a polynomial with roots outside the support of $\mathrm{d}\mu$. Christoffel's research  aimed to comprehend the distribution of zeros, the nodes in such quadrature rules \cite{Uva}. The resulting Christoffel formula, is documented in various classical textbooks on orthogonal polynomials such as \cite{Chi,Sze,Gaut}. 
For a concise  fresher overview of Christoffel and Geronimus, interested readers may refer to \cite{manas}.

These transformations extend beyond measures to encompass a broader setting involving linear functionals \cite{ahi,Chi,Sze}. For a moment linear functional $u$, its canonical or elementary Christoffel transformation involves defining a new moment functional $\hat{u} = (x - a)u$, where $a \in \mathbb{R}$,  \cite{Bue1, Chi, Yoon}. Conversely, the right inverse of a Christoffel transformation is termed the Geronimus transformation. In simpler terms, given a moment linear functional $u$, its elementary or canonical Geronimus transformation yields a new moment linear functional $\check{u}$ satisfying $(x - a)\check{u} = u$. Notably, $\check{u}$ depends on a free parameter in this case  \cite{Geronimus,Maro}. Furthermore, a general Christoffel transformation's right inverse is known as a multiple Geronimus transformation  \cite{DereM}.

These transformations collectively fall under the umbrella of Darboux transformations, a term first coined in the context of integrable systems \cite{matveev}. Gaston Darboux explicitly treated these transformations in 1878 while studying the Sturm–Liouville theory  \cite{darboux2,moutard}.
In the framework of orthogonal polynomials on the real line, factorization of Jacobi matrices similar to these transformations has been investigated  \cite{Bue1,Yoon}. They also play a significant role in analyzing bispectral problems \cite{gru,gru2}.

A crucial aspect of canonical Christoffel transformations lies in their relation to $LU$ factorization (and its flipped version, $UL$ factorization) of the Jacobi matrix. This factorization emerges from a three-term recurrence relation satisfied by a sequence of monic polynomials associated with a nontrivial probability measure $\mu$. Such factorization facilitates the derivation of another Jacobi matrix $\hat{J}$ and its corresponding sequence of monic polynomials $\{\hat{P}_{n}(x)\}_{n=0}^{\infty}$, which are orthogonal with respect to the canonical Christoffel transformation of the measure $\mu$.

Additionally, for a moment linear functional $u$, the Markov--Stieltjes function $S(x)$ plays a pivotal role in orthogonal polynomial theory. It maintains a close relationship with the measure associated with $u$ as well as its rational Padé approximation \cite{Bre,Karl}. When considering the canonical Christoffel transformation $\hat{u}$ of the linear functional $u$, its Stieltjes function $\hat{S}(x) = (x - a)S(x) - u_0$, representing a specific case of spectral linear transformations \cite{Zhe}.

In a series of papers, we have delved into the intricacies of Christoffel and Geronimus transformations within the realm of matrix polynomials. Our exploration commenced with \cite{AAGMM}, where we scrutinized Christoffel transformations applied specifically to monic matrix orthogonal polynomials. Expanding upon this groundwork, our research continued in \cite{AGMM}, where we not only discussed Geronimus transformations within the matrix setting but also introduced spectral techniques for monic perturbations. While the non-monic case was also addressed, it was approached without utilizing spectral techniques. Lastly, in \cite{AGMM2}, we explored  the Geronimus–Uvarov scenario, delving deep into its implications, particularly focusing on its applications to non-Abelian Toda lattices.

In addition, our research extended to \cite{bfm}, where for non mixed multiple orthogonal polynomials we delved into the theory concerning the Christoffel and Geronimus perturbations of two weights. Moreover,  connection formulas between type II multiple orthogonal polynomials, type I linear forms, and vector Stieltjes–Markov vector functions were presented. The perturbation matrix polynomials  were not necessarily monic but  belonged to a restricted class.

In this paper, we investigate  general Christoffel perturbations of mixed multiple orthogonal polynomials. Unlike previous works, we consider a rectangular matrix of measures, allowing for a broad class of perturbations that accommodate Christoffel formulas. Our discussion extends beyond the scope of \cite{AAGMM}, as we permit any rank for the leading coefficient of the polynomial perturbation. Furthermore, our analysis surpasses that of \cite{bfm}, not only due to our exploration of the mixed case but also because our leading and subleading coefficients encompass a broader range, including those previously considered in \cite{bfm}. Importantly, we establish a condition—based on determinants of non-perturbed polynomials and spectral data of the perturbation—that guarantees the existence of perturbed mixed multiple orthogonal polynomials. Unlike the solely necessary condition obtained in \cite{bfm},  we provide her both necessary and sufficient conditions for this existence.

The paper is structured as follows: We commence this introduction by covering some foundational concepts concerning mixed multiple orthogonal polynomials and the spectral theory of matrix polynomials. In the subsequent section, we delve into the form of our perturbation matrix polynomial, which performs right multiplication on the matrix of measures. 

Following this, we present a simple yet non-trivial example that surpasses the framework of \cite{bfm}, showcasing the ideas generalized in \S \ref{Section: General}. Here, we discuss a very general perturbation with simple eigenvalues and provide corresponding Christoffel formulas in Theorem \ref{Theorem:Christoffel_Formulas}. Additionally, we address the necessary modifications required to consider the case with non-simple eigenvalues in the spectrum of the matrix perturbation. 
Moreover, we discuss left multiplication in this section, offering a comprehensive view of perturbations from both sides.

Lastly, in \S \ref{S:Criteria}, we unveil in Theorem \ref{Theorem: Criteria} the characterization for the existence of perturbed mixed multiple orthogonality. This characterization is articulated in terms of determinants of the original orthogonal polynomials evaluated in the spectrum of the perturbation matrix polynomial.

	
	\subsection{Mixed Multiple Orthogonal Polynomials}

	Let's delve into the scenario of a rectangular $q\times p$ matrix of measures:
	\begin{align*}
		\d	\mu=\begin{bNiceMatrix}
			\d\mu_{1,1}&\Cdots &\d\mu_{1,p}\\
			\Vdots & & \Vdots\\
			\d	\mu_{q,1}&\Cdots &\d\mu_{q,p}
		\end{bNiceMatrix},
	\end{align*}
	where the measures \(\mu_{b,a}\) are  supported on the interval \(\Delta \subseteq \mathbb{R}\).

	For \(r \in \mathbb{N}\coloneq\{1,2,3,\dots\}\), we consider the matrix of monomials:
	\[
	X_{[r]}(x) = 
	\begin{bNiceMatrix}
		I_r      \\
		xI_r     \\
		x^2 I_r  \\
		\Vdots
	\end{bNiceMatrix}.
	\]
	
	The moment matrix is defined as:
	\[
	\mathscr{M}\coloneqq \int_{\Delta} X_{[q]}(x) \, \mathrm{d}\mu(x) \, X_{[p]}^\top(x).
	\]

	If all the leading principal submatrices \(\mathscr{M}^{[k]}\) are nonsingular, then the Gauss--Borel factorization exists:
	\[
	\mathscr{M} = \mathscr{L}^{-1} \mathscr{U}^{-1},
	\]
	where \(\mathscr{L}\) is a nonsingular lower triangular semi-infinite matrix and \(\mathscr{U}\) is a nonsingular upper triangular matrix. 	It's important to note that this factorization is not unique due to the freedom:
	\[
	\mathscr{L} \to \mathscr{d}^{-1}\mathscr{L}, \quad \mathscr{U} \to \mathscr{U}\mathscr{d},
	\]
	where \(\mathscr{d}\) is any nonsingular diagonal matrix.
	
	Each choice of the invertible diagonal matrix \(\mathscr{d}\) results in a distinct factorization. Two important normalizations are:
	\begin{enumerate}
		\item The left normalization involves setting \(\mathscr{L}\) as a lower unitriangular matrix. When applicable, we will represent the corresponding triangular matrices as \(S\) and \(\mathscr{U}_L\).
		\item The right normalization involves setting \(\mathscr{U}\) as an upper unitriangular matrix. When applicable, we will denote the corresponding triangular matrices as \(\mathscr{L}_R\) and \(\bar S^\top\).
	\end{enumerate}
	
	Note that the Gauss--Borel factorization can be uniquely expressed as:
\begin{equation}\label{eq:Gauss-Borel}
		\mathscr{M} = S^{-1} H\bar S^{-\top},
\end{equation}

	in terms of lower unitriangular matrices $S,\bar S$ and a nonsingular diagonal matrix $H$. 
	
	
	Associated with the Gauss--Borel  factorization, let's consider the following matrix polynomials:
	\[
	\begin{aligned}
		B(x) &= \mathscr{L} X_{[q]}(x), & A(x) &= X_{[p]}^\top(x) \mathscr{U}.
	\end{aligned}
	\]
	The normalization used in this paper is that the $B$ are monic, i.e.
		\[
	\begin{aligned}
		B(x) &= S X_{[q]}(x), & A(x) &= X_{[p]}^\top(x) \bar S^\top H^{-1}.
	\end{aligned}
	\]
	We represent these matrices in terms of their polynomial entries as follows:
	\[
	\begin{aligned}
		B &= \begin{bNiceMatrix}
			B^{(1)}_0 & \Cdots & B^{(q)}_0 \\[2pt]
			B^{(1)}_1 & \Cdots & B^{(q)}_1 \\[2pt]
			B^{(1)}_2 & \Cdots & B^{(q)}_2 \\
			\Vdots[shorten-end=-0pt] & & \Vdots[shorten-end=-0pt]
		\end{bNiceMatrix}, &
		A &= \left[\begin{NiceMatrix}
			A^{(1)}_0 & A^{(1)}_1 & A^{(1)}_2 & \Cdots \\
			\Vdots & \Vdots & \Vdots & \\
			A^{(p)}_0 & A^{(p)}_1 & A^{(p)}_2 & \Cdots
		\end{NiceMatrix}\right].
	\end{aligned}
	\]
	We have the following relations:
	\[
	\int_{\Delta} B(x) \, \mathrm{d}\mu(x) \, A(x) = I,
	\]
	whose entries are the biorthogonality relations:
	\[
	\int_{\Delta} \sum_{b=1}^q \sum_{a=1}^p B^{(b)}_n(x) \, \mathrm{d}\mu_{b,a}(x) \, A^{(a)}_m(x) = \delta_{n,m}.
	\]
	
	From the Gauss--Borel factorization, it also follows that:
	\begin{align*}
		\int_\Delta B(x) \, \mathrm{d}\mu(x) X_{[p]}^\top(x) &= \mathscr{U}^{-1}, \\
		\int_\Delta X_{[p]}(x) \, \mathrm{d}\mu(x) A(x) &= \mathscr{L}^{-1},
	\end{align*}
	which, by entries, represent the following quasi-diagonal mixed multiple orthogonality relations:
	\begin{align*}
		\int_\Delta x^l \sum_{a=1}^p \mathrm{d}\mu_{b,a}(x) A_n^{(a)}(x) &= 0, 
		& \begin{aligned}
			& b \in \{1, \dots, q\}, \quad l \in \left\{0, \dots, \left\lceil\frac{n-b+2}{q}\right\rceil-1\right\},
		\end{aligned} \\
		\int_\Delta \sum_{b=1}^q B_n^{(b)}(x) \mathrm{d}\mu_{b,a}(x) x^l &= 0, 
		& \begin{aligned}
			& a \in \{1, \dots, p\}, \quad l \in \left\{0, \dots, \left\lceil\frac{n-a+2}{p}\right\rceil-1\right\}.
		\end{aligned}
	\end{align*}
	
	
	For \(r \in \mathbb{N}_0\), the shift block matrix is:
	\[
	\Lambda_{[r]} \coloneq 
	\left[
	\begin{NiceMatrix}
		0_r & I_r & 0_r & \Cdots \\
		0_r & 0_r & I_r & \Ddots \\
		0_r & 0_r & 0_r & \Ddots \\
		\Vdots & \Ddots[shorten-end=3pt] & \Ddots[shorten-end=7pt] & \Ddots[shorten-end=9pt]
	\end{NiceMatrix}
	\right],
	\]
	and for \(r = 1\) we denote \(\Lambda_{[1]}\) as \(\Lambda\). Note that \(\Lambda_{[r]} = \Lambda^r\). These shift matrices have the important property:
	\[
	\Lambda_{[r]}X_{[r]}(x) = x X_{[r]}(x).
	\]
	
	The moment matrix \(\mathscr{M}\) possesses a Hankel-type symmetry relation that can be expressed as:
	\[
	\Lambda_{[q]}\mathscr{M} = \mathscr{M}\Lambda_{[p]}^\top.
	\]
	From this relation and the Gauss--Borel factorization, we derive:
	\begin{align}\label{eq:banded_recurrence_matrix}
		T = \mathscr{L} \Lambda_{[q]} \mathscr{L}^{-1} = \mathscr{U}^{-1} \Lambda_{[p]}^\top \mathscr{U}.
	\end{align}
	Therefore, the  matrix \(T\) is a \((p,q)\)-banded matrix, with \(p\) subdiagonals and \(q\) superdiagonals. Moreover, the following relations are satisfied
	\[
	T B(x) = x B(x), \quad A(x) T = x A(x),
	\]
	meaning \(B\) and \(A\) are right and left eigenvectors, respectively. These equations represent recurrence relations among the mixed multiple orthogonal polynomials. Therefore, \(T\) is known as the recurrence matrix.
	
	 The Christoffel--Darboux (CD) kernel polynomial  $K^{[n]}(x,y)$ are  defined by:
	\begin{align} \label{CDkernel}
		K^{[n]}(x,y) & = A^{[n]}(x)B^{[n]}(y) = \begin{bNiceMatrix}
			A_0^{(1)}(x) & \Cdots & A_{n-1}^{(1)}(x) \\
			\Vdots & & \Vdots \\
			A_0^{(p)}(x) & \Cdots & A_{n-1}^{(p)}(x)
		\end{bNiceMatrix} \begin{bNiceMatrix}
			B_0^{(1)}(y) & \Cdots & B_0^{(q)}(y) \\
			\Vdots & & \Vdots \\
			B_{n-1}^{(1)}(y) & \Cdots & B_{n-1}^{(q)}(y) 
		\end{bNiceMatrix},\\
		K^{[n]}_{a,b}(x,y) & = \sum_{i=0}^{n-1}A^{(a)}_i(x)B_i^{(b)}(y).
	\end{align}
	
 Lastly, it's noteworthy to mention that when $q=p$, mixed multiple orthogonality encompasses matrix orthogonality.
	
	\subsection{Canonical Set of Jordan Chains and Divisibility for Matrix Polynomials} \label{S:CSoJC}
		
Following \cite{MatrixPoly}, we present herein some fundamental results regarding matrix polynomials, pivotal for our subsequent analysis.
	 Our focus lies on  regular matrix   polynomials, i.e. matrix polynomials
	\[ \begin{aligned}
		R(x) & = \sum_{l=0}^N R_l x^l , & R_l &\in \mathbb{C}^{p \times p} ,
	\end{aligned} \]
	such that its determinant it is not identically zero.  The leading coefficient $R_N$ needs not be the identity or even invertible. Consequently, for the degree of $\det R(x)$  we find
\[
\begin{aligned}
\deg \det R(x)&=Np - r,  &r &\in \{0,\dots,Np-1\}.
\end{aligned}
 \]
 The eigenvalues of the matrix polynomial $R(x)$ is by definition the set of the zeros of $\det R(x)$.

As our study progresses, we will impose additional constraints on these matrix polynomials.
	\begin{Proposition}[Smith Form]\label{SmithForm}
	Each matrix polynomial  can be expressed in a form known as the Smith form:
	\[ R(x) = E(x)D(x)F(x), \]
	where $E(x)$ and $F(x)$ possess constant determinants, and $D(x)$ represents a diagonal matrix polynomial. Notably, $D(x)$ exhibits the structure:
	\[ 
	D(x) = \diag \left(
		\prod_{i=1}^{M}(x-x_i)^{\kappa_{i,1}}, \prod_{i=1}^{M}(x-x_i)^{\kappa_{i,2}}, \dots , \prod_{i=1}^{M}(x-x_i)^{\kappa_{i,p}}\right),
		\]
	where $x_i$, with $i\in\{1,\cdots, M\}$, denote the zeros of $\det R(x)$. Here, $\kappa_{i,j}$ represents partial multiplicities. If $K_i$ denotes the multiplicity of $x_i$ as a zero of $\det R(x)$, then:
	\[ 
\begin{aligned}
		K_i &= \sum_{j=1}^{p} \kappa_{i,j}, & \sum_{i=1}^{M}\sum_{j=1}^p \kappa_{i,j} &= Np - r,
\end{aligned}
	 \]
	where $M$ signifies the count of distinct roots of $\det R(x)$, and the partial multiplicities may assume zero values in certain scenarios.
	\end{Proposition}
Moving forward, for the sake of simplicity in notation, we will focus on a single eigenvalue $x_0$ with a multiplicity of $K$.
	\begin{Definition}[Jordan Chains]
\begin{enumerate}
	\item 		A Jordan chain of $R(x)$ associated with $x_0 \in \mathbb{C}$ comprises a set of $L+1$ vectors that adhere to the relation
	\[	\begin{aligned}
			\sum_{l=0}^i \frac{1}{l!}\boldsymbol{v}_{L-l}R^{(l)}(x_0) & = 0, & i &\in \{0, \dots ,L\} .
		\end{aligned}\]
	\item 
	A canonical set of Jordan chains of $R(x)$ corresponding to $x_0$ consists of $K_i$ vectors, organized as follows:
	\[ 
	\{ \boldsymbol{v}_{1,0},\boldsymbol{v}_{1,1}, \dots, \boldsymbol{v}_{1,\kappa_1-1}, \boldsymbol{v}_{2,0}, \boldsymbol{v}_{2,1}, \dots,\boldsymbol{v}_{2,\kappa_2-1}, \dots,\: \boldsymbol{v}_{s,0},\boldsymbol{v}_{s,1}, \cdots, \boldsymbol{v}_{s,\kappa_s-1} \}, 
	\]
	where $s \leq p$ and $\sum_{i=1}^s \kappa_{i} = K$. In this arrangement, each subset of vectors \[\{\boldsymbol{v}_{i,0},\: \boldsymbol{v}_{i,1}, \cdots, \boldsymbol{v}_{i,\kappa_i-1} \}\] constitutes a Jordan chain of length $\kappa_i$. Crucially, the vectors $\boldsymbol{v}_{i,0}$, where $i\in\{1,\cdots,s\}$, are linearly independent.		
\end{enumerate}
	\end{Definition}
Lastly, let's introduce a theorem concerning the divisibility of matrix polynomials, see
	\cite[Corollary 7.11., pag. 203]{MatrixPoly}.
	\begin{Theorem}[Matrix Polynomials Divisibility]\label{Theorem}
Let us consider two regular $p \times p$ matrix polynomials $R(x)$ and $A(x)$. Then, $R(x)$ is a right/left divisor of $A(x)$ if and only if each Jordan chain of $R(x)$ coincides with a Jordan chain of $A(x)$ having the same eigenvalue.
	\end{Theorem}
	\section{The Matrix Structure of the Polynomial Perturbation}
We'll examine  matrix polynomials  
\[	\begin{aligned}
	R(x) & = \sum_{l=0}^N  R_{l}x^l, & R_{l} &  \in \mathbb{C}^{p\times p}, 
\end{aligned}\]
with  leading and sub-leading matrices taking the form:
	\begin{align} \tag{C1}\label{LeadingMatrixConditions}
		R_{N}  & = \begin{bNiceArray}{cw{c}{1cm}c||w{c}{2cm}c}
	\Block{3-2}{0_{(p-r)\times r} }	&	& \Block{3-3}{\left[ t_{N} \right]_{(p-r)\times (p-r)}} &&\\\\\\
		\Block{2-2}{	0_{r\times r} }& &\Block{2-3}{0_{r \times (p-r)}} &&\\\\
		\end{bNiceArray}, & R_{N-1}  & = \begin{bmatrix}
			\left[ R^1_{N-1} \right]_{(p-r)\times r} & \left[ R^2_{N-1} \right]_{(p-r)\times (p-r)} \\ \\
			\left[ t_{N-1} \right]_{r\times r} & \left[ R^4_{N-1} \right]_{r \times (p-r)}
		\end{bmatrix},
	\end{align}
	where $r$ can take values in $\{0,\cdots,p-1\}$, and $\left[ t_{N} \right]_{(p-r)\times (p-r)}$ as well as $\left[ t_{N-1} \right]_{r\times r}$ are upper triangular matrices with nonzero determinant.
	\begin{Proposition} \label{prop1}
	Consider three matrix polynomials, $R(x)$, $\hat{R}(x)$, and $\mathcal{A}(x)$, where $R(x)$ and $\hat{R}(x)$ share the same degree and meet the conditions \eqref{LeadingMatrixConditions}. If these matrices satisfy:
	\begin{equation}\label{eq:divisibility}
		\hat{R}(x) = R(x)\mathcal{A}(x) 
	\end{equation}
	then, $\mathcal{A}(x)$ must be a degree-zero matrix, upper triangular, and possess a nonzero determinant.
	\end{Proposition}
	\begin{proof}
We can express the matrix polynomial  Equation \eqref{eq:divisibility} in terms of its coefficients as follows:
		\[	
		\hat{R}(x) = \sum_{l=0}^N \hat{R}_{l}x^l =R(x)\mathcal{A}(x) = \sum_{l=0}^N\sum_{m=0}^M R_{l}A_m x^{l+m}.
		\]
		If $M \geq 1$, the resulting leading matrix, at order $x^{N+M}$, will satisfy: 
		\begin{align*}
			R_NA_M & = \begin{bmatrix}
				0_{(p-r)\times r} & \left[ t_N \right]_{(p-r)\times (p-r)} \\ \\
				0_{r\times r} & 0_{r \times (p-r)}
			\end{bmatrix} \begin{bmatrix}
				\left[ A^1_{M} \right]_{r\times r} & \left[ A^2_{M} \right]_{r \times (p-r)} \\ 
				\left[ A^3_{M} \right]_{(p-r) \times r} & \left[ A^4_{M} \right]_{(p-r)\times (p-r)}
			\end{bmatrix} \\ \
			& = \begin{bmatrix}
				\left[ t_N \right]_{(p-r)\times (p-r)}\left[ A^3_{M} \right]_{(p-r)\times r} & \left[ t_N \right]_{(p-r)\times (p-r)}\left[ A^4_{M} \right]_{(p-r)\times (p-r)} \\ 
				0_{r\times r} & 0_{r\times (p-r)}
			\end{bmatrix}  \\ 
			& = \begin{bmatrix}
				\left[ t_N A^3_{M} \right]_{(p-r) \times r} & \left[ t_N A^4_{M} \right]_{(p-r)\times (p-r)} \\ \\
				0_{r\times r} & 0_{r \times (p-r)}
			\end{bmatrix} = 0 .
		\end{align*}
		Consequently,
		\[ \left[ A^3_{M} \right]_{(p-r) \times r},\left[ A^4_{M} \right]_{(p-r)\times (p-r)} = 0. \]
		The sub-leading matrix, at order $x^{N+M-1}$, satisfies:
		\begin{align*}
			R_{N-1}A_M + R_{N}A_{M-1} & = 0 \: \: \text{if}\: \: M>1 ,\\
			R_{N-1}A_M + R_{N}A_{M-1} & = \hat{R}_N \: \: \text{if}\: \: M=1. 
		\end{align*}
	In both cases, the final $r$ rows of the resultant matrix are zero, and consequently, the product $R_{N}A_{M-1}$ will yield another matrix with the last $r$ rows also being zero. Therefore,
		\begin{align*}
			R_{N-1}A_M & = \begin{bmatrix}
				\left[ R^1_{N-1} \right]_{(p-r)\times r} & \left[ R^2_{N-1} \right]_{(p-r)\times (p-r)} \\ \\
				\left[ t_{N-1} \right]_{r\times r} & \left[ R^4_{N-1} \right]_{r \times (p-r)}
			\end{bmatrix}\begin{bmatrix}
				\left[ A^1_{M} \right]_{r\times r} & \left[ A^2_{M} \right]_{r \times (p-r)} \\ \\
				0_{(p-r) \times r} & 0_{(p-r)} 
			\end{bmatrix} \\
			& = \begin{bmatrix}
				\left[ R^1_{N-1}A_M^1 \right]_{(p-r) \times r} & \left[ R^1_{N-1}A_M^2 \right]_{(p-r)\times (p-r)} \\ \\
				\left[ t_{N-1} A^1_M\right]_{r\times r} & \left[ t_{N-1} A^2_M\right]_{r \times (p-r)}
			\end{bmatrix} = \begin{bNiceMatrix}
		 \ast & \Cdots & \ast \\ \Vdots & & \Vdots \\ &(p-r) \times p & \\&&\\
					\ast & \Cdots & \ast\\ \\
				&0_{r \times p}&\\
				&&
			\end{bNiceMatrix}.
		\end{align*}
	It follows that $\left[ A^1_M\right]_{r\times r}$ and $\left[ A^2_M\right]_{r \times (p-r)}$ are both zero. Since $A_M = 0$ for $M \geq 1$, $M$ must equal zero. At this point, it could have been more straightforward to express this result in terms of invertible matrices rather than utilizing the upper triangular matrices in the upper-right block of the leading matrix and in the lower-left block of the subleading matrix. Now, we will explicitly leverage the fact that they are both upper triangular and invertible.
		
		The fact that $A_0$ is upper triangular follows from:
		\begin{align*}
			R_NA_0 & = \begin{bmatrix}
				0_{(p-r)\times r} & \left[ t_N \right]_{(p-r)\times (p-r)} \\ \\
				0_{r\times r} & 0_{r \times (p-r)}
			\end{bmatrix} \begin{bmatrix}
				\left[ A^1_{0} \right]_{r\times r} & \left[ A^2_{0} \right]_{r \times (p-r)} \\ \\
				\left[ A^3_{0} \right]_{(p-r)\times r} & \left[ A^4_{0} \right]_{(p-r)\times (p-r)}
			\end{bmatrix}  \\
			& = \begin{bmatrix}
				\left[ t_N A^3_{0} \right]_{(p-r)\times r} & \left[ t_N A^4_{0} \right]_{(p-r)\times (p-r)} \\ \\
				0_{r\times r} & 0_{r \times (p-r)} 
			\end{bmatrix} = \begin{bmatrix}
				0_{(p-r)\times r} & \left[ \hat{t}_{N} \right]_{(p-r)\times (p-r)} \\ \\
				0_{r\times r} & 0_{r \times (p-r)}
			\end{bmatrix} = \hat{R}_N,
		\end{align*}
		which shows that $\left[ A^3_{0} \right]_{(p-r)\times r} = 0$ and $\left[ A^4_{0} \right]_{(p-r)\times (p-r)}=\left[ t_N^{-1} \hat{t}_{N} \right]_{(p-r)\times (p-r)}$. Thus, $\left[ A^4_{0} \right]_{(p-r)\times (p-r)}$ is an upper triangular matrix with nonzero determinant. For the next degree, we have the relations: 
		\begin{align*}
			R_{N-1}A_0 & = \begin{bmatrix}
				\left[ R^1_{N-1} \right]_{(p-r)\times r} & \left[ R^2_{N-1} \right]_{(p-r)\times (p-r)} \\ \\
				\left[ t_{N-1} \right]_{r\times r} & \left[ R^4_{N-1} \right]_{r \times (p-r)}
			\end{bmatrix} \begin{bmatrix}
				\left[ A_0^1 \right]_{r\times r} & \left[ A_0^2\right]_{r \times (p-r)} \\ \\
				0_{(p-r) \times r} & \left[ \hat{t}_{N} \right]_{(p-r)\times (p-r)}
			\end{bmatrix}  \\
			& = \begin{bmatrix}
				\left[ \ast \right]_{(p-r) \times r} & \left[ \ast \right]_{(p-r)\times (p-r)} \\ \\
				\left[ t_{N-1} A_0^1 \right]_{r\times r} & \left[ \ast \right]_{r \times (p-r)}
			\end{bmatrix} = \begin{bmatrix}
				\left[ \hat{R}^1_{N-1} \right]_{(p-r)\times r} & \left[ \hat{R}^2_{N-1} \right]_{(p-r)\times (p-r)} \\ \\
				\left[ \hat{t}_{N-1}\right]_{r\times r} & \left[ \hat{R}^4_{N-1} \right]_{r \times (p-r)}
			\end{bmatrix} = \hat{R}_{N-1}.
		\end{align*} It follows that $\left[ A_0^1 \right]_{r\times r} = \left[ t_{N-1}^{-1}\hat{t}_{N-1}\right]_{r\times r}$ is an upper triangular matrix with nonzero determinant. With all this, the result is proved.
	\end{proof}
	
	\begin{Corollary} \label{Corollary}
		If $\left[ \hat{t}_N \right]$ and/or $\left[ \hat{t}_{N-1} \right]$, which are the leading and sub-leading matrices, respectively, of the matrix polynomial $\hat{R}(x)$, have a zero  entry on the diagonal, the matrix $\mathcal{A}$ will be upper triangular with some zero element on the diagonal, and hence non-invertible. 
	\end{Corollary}
	\begin{proof}
		In the preceding proof, we arrived at $\left[ A^4_{0} \right]_{(p-r)\times (p-r)}=\left[ t_N^{-1} \hat{t}_{N} \right]_{(p-r)\times (p-r)}$ and $\left[ A^1_{0} \right]_{r\times r}=\left[ t_N^{-1} \hat{t}_{N} \right]_{r\times r}$, where the matrices were invertible. Under the given assumptions, the matrices $\left[ \hat{t}_N \right]$ and/or $\left[ \hat{t}_{N-1} \right]$ will not be invertible, and thus $\left[ A^4_{0} \right]_{(p-r)\times (p-r)}$ and/or $\left[ A^1_{0} \right]_{r\times r}$ will also not be invertible. The only possibility is that some element on the diagonal of $\mathcal{A}$ becomes zero. 
	\end{proof}
	
	From now on, we will work with matrix polynomials whose leading and sub-leading matrices, $\left[ t_{N} \right]_{(p-r)\times (p-r)}$ and $\left[ t_{N-1} \right]_{r\times r}$, are the identity matrix:
	\begin{align} \tag{C2}\label{CondicionesMatricesLideresFinales}
		R_{N}  & = \begin{bmatrix}
			0_{(p-r)\times r} & I_{(p-r)\times (p-r)} \\ \\
			0_{r\times r} & 0_{r \times (p-r)}
		\end{bmatrix}, & R_{N-1}  & = \begin{bmatrix}
			\left[ R^1_{N-1} \right]_{(p-r)\times r} & \left[ R^2_{N-1} \right]_{(p-r)\times (p-r)} \\ \\
			I_{r\times r} & \left[ R^4_{N-1} \right]_{r \times (p-r)}
		\end{bmatrix}.
	\end{align}
	\begin{Remark}
		As a particular case of this Proposition, it is easy to see that any matrix polynomial as in condition \eqref{LeadingMatrixConditions} can be obtained as the product of $R(x)$ by another upper triangular matrix with nonzero determinant. Although we will only work with perturbations whose leading matrices satisfy the condition \eqref{CondicionesMatricesLideresFinales}, if subsequently the weight matrix is multiplied by a matrix with a nonzero determinant, orthogonality will still exist, and the newly perturbed polynomials will be a linear combination of the previous ones. 
	\end{Remark}
	\begin{Proposition} \label{Prop2}
		The determinant of a matrix polynomial whose leading and sub-leading matrices satisfy conditions \eqref{CondicionesMatricesLideresFinales} is a polynomial of degree $Np-r$.
	\end{Proposition}
	\begin{proof}
		Let us expand the determinant:
		\begin{align*}
			\det  R(x) & =  \begin{vmatrix}
				\left[ R^1_{N-1} \right]x^{N-1}+O(x^{N-2}) & x^{N}I_{(p-r)} + O(x^{N-1}) \\
				\\
				x^{N-1}I_{r} + O(x^{N-2}) & \left[ R^3_{N-1} \right]x^{N-1}+O(x^{N-2})
			\end{vmatrix} \\
			& = \sum_{\sigma \in \mathcal{S}_p} \text{sgn}(\sigma) R_{1\sigma(1)}R_{2\sigma(2)} \cdots R_{p\sigma(p)} \\&
				= \text{sgn}(\Tilde{\sigma}) R_{1\Tilde{\sigma}(1)}R_{2\Tilde{\sigma}(2)} \cdots R_{(p-r)\Tilde{\sigma}(p-r)} \cdots R_{p\Tilde{\sigma}(p)} 
			+ \sum_{\sigma \neq \Tilde{\sigma}} \text{sgn}(\sigma) R_{1\sigma(1)}R_{2\sigma(2)} \cdots R_{p\sigma(p)},
		\end{align*}
		where the permutation $\Tilde{\sigma}(i)$ is such that $i \rightarrow r+i$ for $i \leq (p-r)$ and $i \rightarrow i-(p-r)$ for $i \geq (p-r)+1$. This term in the determinant expansion gives a contribution of the form:
		\begin{multline*}
			\sgn(\Tilde{\sigma}) R_{1\Tilde{\sigma}(1)}R_{2\Tilde{\sigma}(2)} \cdots R_{(p-r)\Tilde{\sigma}(p-r)} \cdots R_{p\Tilde{\sigma}(p)} \\\begin{aligned}
				&= \sgn(\Tilde{\sigma}) R_{1,(r+1)}R_{2,(r+2)} \cdots R_{(p-r),p} R_{(p-r+1),1} \cdots R_{pr} =
		\sgn(\Tilde{\sigma}) (x^N)^{(p-r)}(x^{N-1})^{r} = x^{Np-r}.
			\end{aligned}
		\end{multline*} Any other permutation either gives zero or gives terms of lower degree.
	\end{proof}
\begin{Proposition}
		$R\left( \Lambda_{[p]}^\top \right)$ is banded lower triangular matrix that from the $Np-r$ subdiagonal is populated with zeros.
\end{Proposition}
	\begin{proof}
		We have
		\begin{equation*}
			R\left( \Lambda_{[p]}^\top \right) = \left[\begin{NiceMatrix}
				R_0 & 0_p & 0_p & \Cdots \\
				R_1 & R_0 & 0_p & \Cdots \\
				\Vdots & & \Ddots & \\
				R_N & R_{N-1} & \Cdots & R_0 & \Cdots \\
				0_p & R_N & R_{N-1} & \Cdots & R_0 & \Cdots\\
				\Vdots & \Ddots[shorten-end=-10pt] &  \Ddots[shorten-end=-10pt] &  \Ddots[shorten-end=-25pt] &  & \Ddots
			\end{NiceMatrix} \right]
		\end{equation*}
		The block 
		\begin{multline*}
			\begin{bmatrix}
				R_{N-1} & R_{N-2} \\
				R_N & R_{N-1}  
			\end{bmatrix} \\= \begin{bmatrix}
				\left[ R^1_{N-1} \right]_{(p-r)\times r} & & \left[ R^2_{N-1} \right]_{(p-r)\times (p-r)} & & &  \\ 
				& & & & R_{N-2} & \\
				I_{r\times r} & & \left[ R^4_{N-1} \right]_{r \times (p-r)} &  &  & \\ \\
				0_{(p-r)\times r} & & I_{(p-r)} & \left[ R^1_{N-1} \right]_{(p-r)\times r} & & \left[ R^2_{N-1} \right]_{(p-r)\times (p-r)} \\ \\
				0_{r\times r} & & 0_{r \times (p-r)} & I_{r\times r} & & \left[ R^4_{N-1} \right]_{r \times (p-r)}
			\end{bmatrix}
		\end{multline*} has $p-r$ subdiagonals. Up to the matrix $R_{N-1}$, there will be $M-1$ matrices, which sum up to a total of $Np-r$ subdiagonals.
	\end{proof}
	\section{The Christoffel perturbation by examples}
	Let's illustrate an novel family of examples  that do not fit in the cases discussed in \cite{bfm}, which will serve as a guide for further generalization. We will consider multiple orthogonality, that is, $q=1$, and specifically, $p=3$. Let's start by studying the perturbation matrix:
	\begin{equation*}
		R(x)= \begin{bmatrix}
			\frac{b^2}{4} & x & 0 \\
			0 & \frac{b^2}{4} & x \\
			1 & b & \frac{b^2}{4}
		\end{bmatrix} = \begin{bmatrix}
			0 & 1 & 0 \\
			0 & 0 & 1 \\
			0 & 0 & 0 
		\end{bmatrix}x + \begin{bmatrix}
			\frac{b^2}{4} & 0 & 0 \\
			0 & \frac{b^2}{4} & 0 \\
			1 & b & \frac{b^2}{4}
		\end{bmatrix}=R_1x+R_0
	\end{equation*}
	with determinant given by:
	\begin{equation*}
		\det R(x) = x^2 - \frac{b^3}{4}x+\frac{b^6}{64}.
	\end{equation*}
	This polynomial has a double root, $x_0=\frac{b^3}{8}$. One can verify that the vectors $\boldsymbol{v}_1=\begin{bmatrix} -\frac{b}{2} & -\frac{b^2}{4} & \frac{b^3}{8}\end{bmatrix}$ and $\boldsymbol{v}_2=\begin{bmatrix} 0 & 1 & 0 \end{bmatrix}$ satisfy the following relations: 
	\begin{equation*}
		\boldsymbol{v}_1R\left(\frac{b^3}{8}\right) = 0 \: \: \text{and} \: \: \boldsymbol{v}_1R'\left(\frac{b^3}{8}\right) + \boldsymbol{v}_2R\left(\frac{b^3}{8}\right) = 0,
	\end{equation*}
	meaning, these two vectors form a canonical set of Jordan chains for $R(x)$.
	
	Let us consider  multiple orthogonal polynomials of type I and II, which satisfy the orthogonality relations with respect to a vector of measures    $\boldsymbol{w}(x)\d x$,  perturb this vector by the matrix polynomial $R(x)$ and assume that the new family of perturbed polynomials exists.
	Then, it is clear that
	\begin{align*}
		R(x)X_{[3]}^\top(x) & = \left( R_1x+R_0 \right)  \left[\begin{NiceMatrix} I_3 & x I_3 & \Cdots \end{NiceMatrix}\right] \\&= X_{[3]}^\top(x) \left(R_1 (\Lambda_{[3]}^1)^\top + R_0 (\Lambda_{[3]}^0)^\top \right)\\& = X_{[3]}^\top(x) R\Big(\Lambda_{[3]}^\top\Big), \
	\end{align*}
	so that, the perturbed moment matrix will be
	\begin{align*}
		\hat{\mathcal{M}} & = \int X_{[1]}(x)\boldsymbol{w}(x)R(x)X_{[3]}^\top(x)\d x \\&= 
		\int X_{[1]}(x)\boldsymbol{w}(x)X_{[3]}^\top(x)\:dx R(\Lambda_{[3]}^\top) \\& = \mathcal{M}R\Big(\Lambda_{[3]}^\top\Big),
	\end{align*}
	where it is understood that, given a $3 \times 3$ matrix:
	\begin{equation*}
		R\cdot\left[\begin{NiceMatrix} I_3 & x I_3 & \Cdots \end{NiceMatrix}\right] = 
	\left[	\begin{NiceMatrix} R & x R & \Cdots \end{NiceMatrix} \right]= \left[\begin{NiceMatrix} I_3 & x I_3 & \Cdots \end{NiceMatrix}\right]\cdot R .
	\end{equation*}
	Here $\C^{3\times \infty}[x]$ is considered a $\C^{3\times 3}[x]$ bimodule. 
	With all this, and assuming that Gauss--Borel factorization \eqref{eq:Gauss-Borel} exists for both moment matrices, we have:
	\begin{align*}
		\hat{\mathcal{M}} & = \hat{S}^{-1}\hat{H}\hat{\Bar{S}}^{-\top}, \\
		\mathcal{M} & = S^{-1}H\Bar{S}^{-\top}, \\
		\Omega & \coloneq  S\hat{S}^{-1} = H\Bar{S}^{-\top}R\Big(\Lambda_{[3]}^\top\Big)\hat{\Bar{S}}^\top\hat{H}^{-1} .
	\end{align*}
	Due to $S$, $\hat{S}$, $\Bar{S}$, and $\hat{\Bar{S}}$ being lower unitriangular matrices and due to the band structure of $R\Big(\Lambda_{[3]}^\top\Big)$, we deduce that $\Omega$ is of the form: 
	\begin{equation*}
		\Omega =\left[ \begin{NiceMatrix}
			1 & 0 & \Cdots  &&&\\
			\Omega_{1,0} & 1 & 0 & \Cdots&& &\\
			\Omega_{2,0} & \Omega_{2,1} & 1 & 0 & \Cdots&\\
			0 & \Omega_{3,1} & \Omega_{3,2} & 1 &&\\
			\Vdots[shorten-end=-10pt] & \Ddots[shorten-end=-23pt]&\Ddots[shorten-end=-30pt] & \Ddots[shorten-end=-30pt]& \Ddots[shorten-end=-5pt]&\\
			\phantom{i}&	\phantom{i}&	\phantom{i}&	\phantom{i}&	\phantom{i}&
		\end{NiceMatrix}\right].
	\end{equation*}
	We have the following connection formulas:
\begin{Proposition}
		The original and perturbed polynomials are connected trough the formulas:
	\begin{align} 
		\label{ConexAE1}    A(x) \Omega  & = R(x)\hat{A}(x), \\
		\label{ConexBE1}   \Omega \hat{B}(x) & = B(x) .
	\end{align}
\end{Proposition}
	\begin{proof}
		It follows from:
		\begin{align*} 
			A(x)\Omega & = (\Bar{S}X_{[3]}(x))^\top H^{-1} \cdot H \Bar{S}^{-\top} R(\Lambda_{[3]}^\top) \hat{\Bar{S}}^\top \hat{H}^{-1} = (P(\Lambda) X_{[3]}(x))^\top \hat{\Bar{S}}^\top \hat{H}^{-1} = R(x)\hat{A}(x), \\
			\Omega \hat{B}(x) & =S\hat{S}^{-1} \hat{S} X_{[1]}(x) = B(x).
		\end{align*}
	\end{proof}
	Entrywise, Equation \eqref{ConexAE1} takes the form: 
	\begin{equation} \label{CompConexAE1}
		\begin{bmatrix}
			A_n^{(1)}(x) \\[2pt]
			A_n^{(2)}(x) \\[2pt]
			A_n^{(3)}(x) 
		\end{bmatrix}+\begin{bmatrix}
			A_{n+1}^{(1)}(x) \\[2pt]
			A_{n+1}^{(2)}(x) \\[2pt]
			A_{n+1}^{(3)}(x) 
		\end{bmatrix}\Omega_{n+1,n}+\begin{bmatrix}
			A_{n+2}^{(1)}(x) \\[2pt]
			A_{n+2}^{(2)}(x) \\[2pt]
			A_{n+2}^{(3)}(x) 
		\end{bmatrix}\Omega_{n+2,n} = W(x)\begin{bmatrix}
			\hat{A}_n^{(1)}(x) \\[2pt]
			\hat{A}_n^{(2)}(x) \\[2pt]
			\hat{A}_n^{(3)}(x) 
		\end{bmatrix}.
	\end{equation}
	Applying the Jordan chain vectors: 
	\begin{align*}
		\boldsymbol{v}_1A\left(\frac{b^3}{8}\right)\Omega & = \boldsymbol{v}_1R\left(\frac{b^3}{8}\right)\hat{A}\left(\frac{b^3}{8}\right) = 0, \\
		A'(x)\Omega & = R'(x)\hat{A}(x)+R(x)\hat{A}'(x)\: \: \boldsymbol{v}_1A'\left(\frac{b^3}{8}\right)\Omega = \boldsymbol{v}_1R'\left(\frac{b^3}{8}\right)\hat{A}\left(\frac{b^3}{8}\right), \\
		\left(\boldsymbol{v}_1A'\left(\frac{b^3}{8}\right) + \boldsymbol{v}_2A\left(\frac{b^3}{8}\right) \right)\Omega & = \left(\boldsymbol{v}_1R'\left(\frac{b^3}{8}\right)+\boldsymbol{v}_2R\left(\frac{b^3}{8}\right)\right) \hat{A}\left(\frac{b^3}{8}\right) = 0.
	\end{align*}
	These vectors  allow us to solve for the entries $\Omega_{i+1,i}$ and $\Omega_{i+2,i}$. To do this, let's introduce the notation: 
	\begin{align*}
		\mathbb{A}_n^1 & = \left(\boldsymbol{v}_1A\left(\frac{b^3}{8}\right)\right)_{n+1} = \sum_{a=1}^{3}v_{1,a}A_n^{(a)}\left(\frac{b^3}{8}\right), \\
		\mathbb{A}_n^2 & = \left(\boldsymbol{v}_1A'\left(\frac{b^3}{8}\right)+\boldsymbol{v}_2A\left(\frac{b^3}{8}\right)\right)_{n+1} = \sum_{a=1}^{3} \left(v_{1,a}A_n'^{(a)}\left(\frac{b^3}{8}\right)+v_{2,a}A_n^{(a)}\left(\frac{b^3}{8}\right)\right).
	\end{align*} Now, we can apply the Jordan chain vectors to Equation \eqref{CompConexAE1}: 
	\begin{align*}
		 \begin{bmatrix}
			\mathbb{A}_n^1 \\[2pt]
			\mathbb{A}_n^2 
		\end{bmatrix}+\begin{bmatrix}
			\mathbb{A}_{n+1}^1 \\[2pt]
			\mathbb{A}_{n+1}^2 
		\end{bmatrix}\Omega_{n+1,n}+\begin{bmatrix}
			\mathbb{A}_{n+2}^1 \\[2pt]
			\mathbb{A}_{n+2}^2 
		\end{bmatrix}\Omega_{n+2,n}& = 0, \\[2pt]
		\begin{bmatrix}
			\mathbb{A}_n^1 \\
			\mathbb{A}_n^2
		\end{bmatrix} + \begin{bmatrix}
			\mathbb{A}_{n+1}^1 & \mathbb{A}_{n+2}^1 \\[2pt]
			\mathbb{A}_{n+1}^2 & \mathbb{A}_{n+2}^2
		\end{bmatrix} \begin{bmatrix}
			\Omega_{n+1,n} \\[2pt]
			\Omega_{n+2,n}
		\end{bmatrix} &= 0,  \\
		-  \begin{bmatrix}
			\mathbb{A}_{n+1}^1 & \mathbb{A}_{n+2}^1 \\[2pt]
			\mathbb{A}_{n+1}^2 & \mathbb{A}_{n+2}^2
		\end{bmatrix}^{-1}\begin{bmatrix}
			\mathbb{A}_n^1 \\[2pt]
			\mathbb{A}_n^2
		\end{bmatrix} &= \begin{bmatrix}
			\Omega_{n+1,n} \\[2pt]
			\Omega_{n+2,n}
		\end{bmatrix}.
	\end{align*}
	On the other hand, we can obtain connection formulas for the kernel polynomials, see Equation \eqref{CDkernel}. For $n \geq 2$: 
	\begin{equation*}
		R(x) \hat{K}^{[n]}(x,y) = K^{[n]}(x,y) + \begin{bmatrix}
			A_n^{(1)}(x) & A_{n+1}^{(1)}(x) \\[3pt]
			A_n^{(2)}(x) & A_{n+1}^{(2)}(x) \\[3pt]
			A_n^{(3)}(x) & A_{n+1}^{(3)}(x) 
		\end{bmatrix}  \begin{bmatrix}
			\Omega_{n,n-1} & \Omega_{n,n-2} \\
			\Omega_{n+1,n-1} & 0 
		\end{bmatrix}  \begin{bmatrix}
			\hat{B}_{n-1}(y) \\
			\hat{B}_{n-2}(y)
		\end{bmatrix}. 
	\end{equation*}
    For $n=1$: 
    \begin{equation*}
        R(x) \hat{K}^{[1]}(x,y) = K^{[1]}(x,y) + \begin{bmatrix}
			A_1^{(1)}(x) & A_{2}^{(1)}(x) \\[3pt]
			A_1^{(2)}(x) & A_{2}^{(2)}(x) \\[3pt]
			A_1^{(3)}(x) & A_{2}^{(3)}(x) 
		\end{bmatrix}  \begin{bmatrix}
			\Omega_{1,0}\\
			\Omega_{2,0} 
		\end{bmatrix} B_0(y).
    \end{equation*}
	Let's introduce a notation of the form:
	\begin{align*}
		\mathbb{K}^{[n]}(1,y) & = \sum_{a=1}^{3}v_{1,a}K^{[n]}_{a}\left(\frac{b^3}{8}\right), \\
		\mathbb{K}^{[n]}(2,y) & = \sum_{a=1}^{3} \left(v_{1,a}K^{[n]'}_{a}\left(\frac{b^3}{8}\right)+v_{2,a}K^{[n]}_{a}\left(\frac{b^3}{8}\right)\right).
	\end{align*}
	So, acting with the different vectors and evaluating at $x=\frac{b^3}{8}$, this last equation reads: 
	\begin{equation*}
		\begin{bmatrix}
			\mathbb{K}^{[n]}(1,y) \\
			\mathbb{K}^{[n]}(2,y)
		\end{bmatrix} + \begin{bmatrix}
			\mathbb{A}_{n}^1 & \mathbb{A}_{n+1}^1 \\[3pt]
			\mathbb{A}_{n}^2 & \mathbb{A}_{n+1}^2
		\end{bmatrix}\begin{bmatrix}
			\Omega_{n,n-1} & \Omega_{n,n-2} \\
			\Omega_{n+1,n-1} & 0 
		\end{bmatrix}  \begin{bmatrix}
			\hat{B}_{n-1}(y) \\
			\hat{B}_{n-2}(y)
		\end{bmatrix} = 0.
	\end{equation*}
	With all this, and in terms of the $\tau$ determinants:
	\[
	\begin{aligned}
		\tau_n & \coloneq \begin{vmatrix}
			\mathbb{A}_{n}^1 & \mathbb{A}_{n+1}^1 \\[3pt]
			\mathbb{A}_{n}^2 & \mathbb{A}_{n+1}^2   
		\end{vmatrix}, & \tau_{n-1}^{(1)}& \coloneq  \begin{vmatrix}
			\mathbb{A}_{n-1}^1 & \mathbb{A}_{n+1}^1 \\[3pt]
			\mathbb{A}_{n-1}^2 & \mathbb{A}_{n+1}^2  
		\end{vmatrix},
	\end{aligned}
	\]
	we can now obtain:

\begin{Proposition}
Explicit formulas for $\hat{A}(x)$ and $\hat{B}(x)$ are:
		\begin{align*}
		\Omega_{n+1,n-1} & = - \frac{\tau_{n-1}}{\tau_n}, \quad  \Omega_{n,n-1} = \frac{\tau_{n-1}^{(1)}}{\tau_n}, & 
	\end{align*}
	\begin{align*}
		\hat{B}_{n-1}(y) & = \frac{\begin{vmatrix}
				\mathbb{A}_{n}^1 & \mathbb{K}^{[n]}(1,y) \\[3pt]
				\mathbb{A}_{n}^2 & \mathbb{K}^{[n]}(2,y)
		\end{vmatrix}}{\tau_{n-1}}, &
		\left[R(x)\hat{A}(x)\right]_{n+1,a} & = \frac{1}{\tau_n}\begin{vmatrix}
			A_{n}^{(a)}(x) & A_{n+1}^{(a)}(x) & A_{n+2}^{(a)}(x) \\[3pt]
			\mathbb{A}_n^1 & \mathbb{A}_{n+1}^1 & \mathbb{A}_{n+2}^1 \\[3pt]
			\mathbb{A}_n^2 & \mathbb{A}_{n+1}^2 & \mathbb{A}_{n+2}^2
		\end{vmatrix}.
	\end{align*}
\end{Proposition}
	\begin{proof}
		It follows from:
		\begin{align*}
			-  \begin{bmatrix}
				\mathbb{A}_{n}^1 & \mathbb{A}_{n+1}^1 \\[3pt]
				\mathbb{A}_{n}^2 & \mathbb{A}_{n+1}^2
			\end{bmatrix}^{-1}
			\begin{bmatrix}
				\mathbb{K}^{[n]}(1,y) \\[3pt]
				\mathbb{K}^{[n]}(2,y)
			\end{bmatrix} &= \begin{bmatrix}
				\Omega_{n,n-1} & \Omega_{n,n-2} \\[3pt]
				\Omega_{n+1,n-1} & 0 
			\end{bmatrix}  \begin{bmatrix}
				\hat{B}_{n-1}(y) \\
				\hat{B}_{n-2}(y)
			\end{bmatrix}, \\
			- \begin{bmatrix} 0 & 1 \end{bmatrix}\begin{bmatrix}
				\mathbb{A}_{n}^1 & \mathbb{A}_{n+1}^1 \\[3pt]
				\mathbb{A}_{n}^2 & \mathbb{A}_{n+1}^2
			\end{bmatrix}^{-1}
			\begin{bmatrix}
				\mathbb{K}^{[n]}(1,y) \\
				\mathbb{K}^{[n]}(2,y)
			\end{bmatrix}&= \Omega_{n+1,n-1}\hat{B}_{n-1}(y).
		\end{align*}
		Furthermore, $\Omega_{n+1,n-1}$ is obtained from: 
		\begin{equation*}
			-\begin{bmatrix} 0 & 1 \end{bmatrix}\begin{bmatrix}
				\mathbb{A}_{n}^1 & \mathbb{A}_{n+1}^1 \\[3pt]
				\mathbb{A}_{n}^2 & \mathbb{A}_{n+1}^2
			\end{bmatrix}^{-1}\begin{bmatrix}
				\mathbb{A}_{n-1}^1 \\[3pt]
				\mathbb{A}_{n-1}^2
			\end{bmatrix} = \begin{bmatrix} 0 & 1 \end{bmatrix}
			\begin{bmatrix}
				\Omega_{n,n-1} \\
				\Omega_{n+1,n-1}
			\end{bmatrix} = \Omega_{n+1,n-1} = - \frac{\tau_{n-1}}{\tau_n}.
		\end{equation*}
		The last relation comes from the $a$-th entry of Equation \eqref{ConexAE1}: 
		\begin{align*}
			\left[R(x)\hat{A}(x)\right]_{n+1,a} & = \sum_{\Bar{a}=1}^p \left(R(x)\right)_{a,\Bar{a}}\hat{A}_n^{(\Bar{a})}(x) = A_n^{(a)}(x) + \Omega_{n+1,n}A_{n+1}^{(a)}(x) + \Omega_{n+2,n}A_{n+2}^{(a)}(x) \\&= A_n^{(a)}(x) + \begin{bmatrix}
				A_{n+1}^{(a)}(x) & A_{n+2}^{(a)}(x)
			\end{bmatrix} \begin{bmatrix}
				\Omega_{n+1,n} \\
				\Omega_{n+2,n}
			\end{bmatrix} \\
			& = A_n^{(a)}(x) - \begin{bmatrix}
				A_{n+1}^{(a)}(x) & A_{n+2}^{(a)}(x)
			\end{bmatrix} \begin{bmatrix}
				\mathbb{A}_{n+1}^1 & \mathbb{A}_{n+2}^1 \\[3pt]
				\mathbb{A}_{n+1}^2 & \mathbb{A}_{n+2}^2
			\end{bmatrix}^{-1}\begin{bmatrix}
				\mathbb{A}_{n}^1 \\
				\mathbb{A}_{n}^2
			\end{bmatrix}.
		\end{align*}
    The proof for the case $n=1$ is completely analogous and yields the same relation for $\hat{B}_0(y)$
	\end{proof}
\begin{Remark}
    For this simple case, one can obtain explicit formulas for the inverse matrix of $R(x)$. The components $\hat{A}_n^{(a)}(x)$ can be rewritten as follows:  
\begin{align*}
    \hat{A}_n^{(a)}(x) & = \frac{1}{(x^2-\frac{b^3}{4}x+\frac{b^6}{64}) \, \tau_n} \sum_{\Bar{a}=1}^3 \left( \emph{adj} \, R(x) \right)_{a,\Bar{a}}\begin{vmatrix}
			A_{n}^{(\Bar{a})}(x) & A_{n+1}^{(\Bar{a})}(x) & A_{n+2}^{(\Bar{a})}(x) \\[3pt]
			\mathbb{A}_n^1 & \mathbb{A}_{n+1}^1 & \mathbb{A}_{n+2}^1 \\[3pt]
			\mathbb{A}_n^2 & \mathbb{A}_{n+1}^2 & \mathbb{A}_{n+2}^2
		\end{vmatrix} \\
  \text{with} \quad & \emph{adj} \, R(x) = \begin{bmatrix}
        \frac{b^4}{16}-bx & x & -\frac{b^2}{4} \\[3pt]
        -\frac{b^2}{4}x & \frac{b^4}{16} & x-\frac{b^3}{4} \\[3pt]
        x^2 & -\frac{b^2}{4}x & \frac{b^4}{16}
    \end{bmatrix}
\end{align*}
In Section \S \ref{S:Criteria}, it will be shown that a construction of the form $R(x)\hat{A}(x)=A(x)\Omega$ is divisible by $R(x)$, so the decomposition performed is valid. 
\end{Remark}
\begin{Remark}
	The name $\tau$ determinants is motivated because its close relations with  corresponding $\tau$ functions of  associated multiple Toda lattices. We will see later that, provided $\tau_n$ is non-zero, $n\in\N_0$, the new perturbed family of orthogonal polynomials exists.
	\end{Remark}

	\section{General Polynomial Perturbation}\label{Section: General}
	Given a $q \times p$  matrix of measures $\d\mu$ for the original mixed multiple orthogonal polynomials,  let us consider a right multiplication of the form:
	\begin{equation*}
		\d\hat{\mu}(x)= \d\mu(x)  R(x),
	\end{equation*}
	where $R(x)$ is a matrix polynomial that satisfies  \eqref{CondicionesMatricesLideresFinales}. In particular, we will have a matrix polynomial of degree $N$ whose leading matrix has rank $p-r$. Initially, we assume that, given this perturbation, there exists a new family of orthogonal polynomials. With this, 
	
	\begin{Proposition}
	The relationship between the moment matrices is given by:
	\begin{equation}
		\label{MatricesMomentosMP}
		\hat{\mathcal{M}} = \mathcal{M}R(\Lambda^\top).
	\end{equation}
	\end{Proposition}
	\begin{proof}
		Let's start by recalling that
		\begin{equation*}
			R(x)X_{[p]}^\top  = X_{[p]}^\top R\left( \Lambda_{[p]}^\top \right).
		\end{equation*}
		By the definition of moment matrices, we have that,
		\begin{align*}
			\hat{\mathcal{M}} & =  \int X_{[q]}(x) \d\hat{\mu}(x)X_{[p]}^\top(x) \\\
			& =  \int X_{[q]}(x)\d\mu (x)
			R(x)X_{[p]}^\top(x) = \int X_{[q]}(x) \d\mu(x)
			X_{[p]}^\top R\left( \Lambda_{[p]}^\top \right) \\
			& = \mathcal{M} R(\Lambda ^\top )
		\end{align*}
	\end{proof}
	Next, let's assume that there exists a Gauss--Borel factorization for each of the moment matrices, that is:
	\begin{align}
		\label{GaussBorelAMP}   \hat{\mathcal{M}} & = \hat{S}^{-1}\hat{H}\hat{\Bar{S}}^{-\top} \\
		\label{GaussBorelBMP}   \mathcal{M} & = S^{-1}H\Bar{S}^{-\top}.
	\end{align}
\begin{Definition}
	The connection matrix $\Omega$ is given by:
	\begin{align*}
		\Omega\coloneq  S\hat{S}^{-1}.
	\end{align*}
\end{Definition}
\begin{Proposition}
	The matrix  $\Omega$ can be defined alternatively as
		\begin{align*}
		\Omega = S\hat{S}^{-1} = H\Bar{S}^{-\top}R(\Lambda^\top)\hat{\Bar{S}}^\top\hat{H}^{-1}.
	\end{align*}
It is a banded lower  unitriangular matrix; i.e., it has $pN-r$ subdiagonals and the main diagonal filled with ones, and everything else zero. Hence, we have a banded matrix that can also be written as banded $p\times p$ block matrix:
	\begin{align*}
		\Omega &= \begin{bNiceMatrix}
			1 & 0 & \Cdots[shorten-end=7pt] \\
			\Omega_{1,0} & 1 & 0 & \Cdots[shorten-end=7pt]  \\
			\Omega_{2,0} & \Omega_{2,1} & 1 & 0 & \Cdots[shorten-end=7pt] \\
			\Vdots & \Vdots & & \Ddots & \Ddots[shorten-end=8pt]  \\
			\Omega_{pN-r,0} & \Omega_{pN-r,1} & \Cdots & \Omega_{pN-r,pN-r-1} & 1 &  &  \\
			0 & \Omega_{pN-r+1,1} &  &  & \Omega_{pN-r+1,pN-r} & \Ddots[shorten-end=8pt] \\
			\Vdots[shorten-end=1pt] & \Ddots[shorten-end=-25pt] & \Ddots[shorten-end=-45pt]  & \Ddots[shorten-end=-45pt]  &  & \Ddots[shorten-end=-5pt]  \\
		\end{bNiceMatrix} \\&=  \begin{bNiceMatrix}
			\left[\Omega_{0,0}\right]_p& I_p & 0_p& \Cdots[shorten-end=7pt]  \\
			\left[\Omega_{p,0}\right]_p &\left[\Omega_{p,p}\right]_p & I_p & 0_p & \Cdots[shorten-end=7pt] \\
		\Vdots & \Vdots & & \Ddots & \Ddots[shorten-end=8pt]  \\
				\left[\Omega_{pN,0}\right]_p & 		\left[\Omega_{pN,p}\right]_p & \Cdots &\left[\Omega_{pN,pN}\right]_p & I_P &  &  \\
		0_p & \left[\Omega_{pN+p,p}\right]_p &  &  & \left[\Omega_{pN+p,pN}\right]_p  & \Ddots[shorten-end=8pt] \\
		\Vdots[shorten-end=1pt] & \Ddots[shorten-end=-25pt] & \Ddots[shorten-end=-55pt]  & \Ddots[shorten-end=-75pt]  &  & \Ddots[shorten-end=-5pt]  \\
		\end{bNiceMatrix}
	\end{align*}
	\end{Proposition}
	\begin{proof}
		Starting from Equation \eqref{MatricesMomentosMP}, and substituting the explicit expressions in equations \eqref{GaussBorelAMP} and \eqref{GaussBorelBMP} yields the equivalence between both expressions.
		
		The fact that there are only $Np-r$ subdiagonals comes from the structure of the matrix $R(\Lambda^\top)$. When multiplied by upper triangular matrices on both sides, everything below the $Np-r$ subdiagonals will still be filled with zeros.
		
		The main diagonal is populated by ones since both matrices, $S$ and $\hat{S}^{-1}$, are lower unitriangular matrices.
	\end{proof}
	The connection matrix \emph{connects} the original polynomials and the perturbed ones as follows:
\begin{Proposition}
		We have the following connection formulas between the original and perturbed polynomials:
	\begin{align}
		\label{FormulasConexAMP} A(x) \Omega  & = R(x)\hat{A}(x), \\
		\label{FormulasConexBMP} \Omega \hat{B}(x) & = B(x) .
	\end{align}
\end{Proposition}

	\begin{proof}
		It follows immediately from:
		\begin{align*} 
			A(x)\Omega & = (\Bar{S}X_{[p]})^\top H^{-1} \cdot H \Bar{S}^{-\top} \left(R(\Lambda_{[p]}^\top)\right) \hat{\Bar{S}}^\top       \hat{H}^{-1} = X_{[p]}^\top R(\Lambda_{[p]}^\top) \hat{\Bar{S}}^\top \hat{H}^{-1} = R(x)\hat{A}(x) \\
			\Omega \hat{B}(x) & = S\hat{S}^{-1} \hat{S} X_{[q]} = B(x).
		\end{align*}
	\end{proof}
	
	\subsection{Simple Eigenvalues}
	For simplicity in the discussion, although it will be generalized later in the paper, we will assume that all zeros of the determinant of $R(x)$ are simple. From Proposition \ref{Prop2}, we know that the determinant is of degree $pN-r$, so there exist $M = pN-r$ roots ($M$ was used to denote the number of distinct roots of $\det R(x)$ in \S \ref{S:CSoJC}). Therefore, for each root of the determinant, there exists a left eigenvector (one Jordan chain of length 1) such that:
	\begin{equation*}
		\boldsymbol{v}_iR(x_i) = 0.
	\end{equation*}
	We introduce the following notation: 
	\begin{align*}
		\mathbb{A}_n(x_i) & = \left(\boldsymbol{v}_iA(x_i)\right)_{n+1} = \sum_{a=1}^{p}v_{i,a}A_n^{(a)}(x_i), \\
		\mathbb{K}_b^{[n]}(x_i,y) & = \boldsymbol{v}_iK^{[n]}(x_i,y)= \sum_{a=1}^{p}v_{i,a}K_{a,b}^{[n]}(x_i,y), 
	\end{align*}
	where the polynomial $K^{[n]}(x,y)$ are the CD kernel polynomial defined in \eqref{CDkernel}.
	
\begin{Proposition}
	The entries of the connection matrix can be obtained by studying a linear system of equations:
	\begin{equation}
		\label{SistemaQMP} -\begin{bNiceMatrix}
			\mathbb{A}_n(x_1) & \Cdots & \mathbb{A}_{n+M-1}(x_1) \\
			\mathbb{A}_n(x_2) & \Cdots & \mathbb{A}_{n+M-1}(x_2) \\
			\Vdots &  & \Vdots \\
			\mathbb{A}_n(x_M) & \Cdots & \mathbb{A}_{n+M-1}(x_M) 
		\end{bNiceMatrix}^{-1} \begin{bNiceMatrix}
			\mathbb{A}_{n-1}(x_1) \\
			\mathbb{A}_{n-1}(x_2) \\
			\Vdots \\
			\mathbb{A}_{n-1}(x_M)
		\end{bNiceMatrix}= \begin{bNiceMatrix}
			\Omega_{n,n-1} \\
			\Omega_{n+1,n-1} \\
			\Vdots \\
			\Omega_{n+M-1,n-1}
		\end{bNiceMatrix}.
	\end{equation}
\end{Proposition}
	\begin{proof}
		From Equation \eqref{FormulasConexAMP} for the $(n-1)$-th entry:
		\begin{equation*}
			\begin{bNiceMatrix}
				A_{n-1}^{(1)}(x) + \Omega_{n,n-1}A_{n}^{(1)}(x) + \Omega_{n+1,n-1}A_{n+1}^{(1)}(x) + \cdots + \Omega_{n+M-1,n-1}A_{n+M-1}^{(1)}(x) \\[3pt]
				A_{n-1}^{(2)}(x) + \Omega_{n,n-1}A_{n}^{(2)}(x) + \Omega_{n+1,n-1}A_{n+1}^{(2)}(x) + \cdots + \Omega_{n+M-1,n-1}A_{n+M-1}^{(2)}(x) \\
				\Vdots \\
				A_{n-1}^{(p)}(x) + \Omega_{n,n-1}A_{n}^{(p)}(x) + \Omega_{n+1,n-1}A_{n+1}^{(p)}(x) + \cdots + \Omega_{n+M-1,n-1}A_{n+M-1}^{(p)}(x) 
			\end{bNiceMatrix} = R(x) \begin{bNiceMatrix}
				\hat{A}_n^{(1)}(x) \\[3pt]
				\hat{A}_n^{(2)}(x) \\
				\Vdots \\
				\hat{A}_n^{(p)}(x) 
			\end{bNiceMatrix}.
		\end{equation*}
		Evaluating this expression at $x=x_1$ and left multiplying by   the left eigenvector $\boldsymbol{v}_1$ yields:
		\begin{gather*}
			\boldsymbol{v}_i  R(x_1)\begin{bNiceMatrix}
				\hat{A}_n^{(1)}(x_1) \\
				\hat{A}_n^{(2)}(x_1) \\
				\Vdots \\
				\hat{A}_n^{(p)}(x_1) 
			\end{bNiceMatrix} = 0 ,\\
			\mathbb{A}_{n-1}(x_1) + \Omega_{n,n-1}\mathbb{A}_{n}(x_1) + \Omega_{n+1,n-1}\mathbb{A}_{n+1}(x_1) + \cdots + \Omega_{n+M-1,n-1}\mathbb{A}_{n+M-1}(x_1) = 0.
		\end{gather*}
		It is noteworthy here that we have $M$ zeros and hence $M$ left eigenvectors, which coincides with the number of unknown entries per column for the matrix $\Omega$.  Similar equations appear for the other eigenvalues $x_i$. We collect all these equations for $x_i$ with $i \in \{1,\dots,M\}$:
		\begin{equation*}
		\left\{	\begin{aligned}
				\mathbb{A}_{n-1}(x_1) + \Omega_{n,n-1}\mathbb{A}_{n}(x_1) + \Omega_{n+1,n-1}\mathbb{A}_{n+1}(x_1)  + \cdots + \Omega_{n+M-1,n-1}\mathbb{A}_{n+M-1}(x_1)&=0 \\
				\mathbb{A}_{n-1}(x_2) + \Omega_{n,n-1}\mathbb{A}_{n}(x_2) + \Omega_{n+1,n-1}\mathbb{A}_{n+1}(x_2)  + \cdots + \Omega_{n+M-1,n-1}\mathbb{A}_{n+M-1}(x_2)&=0 \\
		&\,\,	\vdots \\
				\mathbb{A}_{n-1}(x_M) + \Omega_{n,n-1}\mathbb{A}_{n}(x_M) + \Omega_{n+1,n-1}\mathbb{A}_{n+1}(x_M)  + \cdots + \Omega_{n+M-1,n-1}\mathbb{A}_{n+M-1}(x_M)&=0.
			\end{aligned} \right.
		\end{equation*}
		Which can be rewritten as an inhomogeneous  linear system:
		\begin{equation*}
			-\begin{bNiceMatrix}
				\mathbb{A}_{n-1}(x_1) \\
				\mathbb{A}_{n-1}(x_2) \\
				\Vdots \\
				\mathbb{A}_{n-1}(x_M)
			\end{bNiceMatrix} = \begin{bNiceMatrix}
				\mathbb{A}_{n}(x_1) & \Cdots & \mathbb{A}_{n+M-1}(x_1) \\
				\mathbb{A}_{n}(x_2) & \Cdots & \mathbb{A}_{n+M-1}(x_2) \\
				\Vdots & & \Vdots \\
				\mathbb{A}_{n}(x_M) & \Cdots & \mathbb{A}_{n+M-1}(x_M) 
			\end{bNiceMatrix} \begin{bNiceMatrix}
				\Omega_{n,n-1} \\
				\Omega_{n+1,n-1} \\
				\Vdots \\
				\Omega_{n+M-1,n-1}
			\end{bNiceMatrix}.
		\end{equation*}
		From which it is easy to arrive at the stated  equation.
	\end{proof}
	
	Let's examine what the commutator, $[\Omega,\Pi_n]$, where $\Pi_n$ is the diagonal matrix with all entries zero but for the $n$ first which equal the unity. We have:
	\begin{equation}
		\label{TruncaciónOmegaMP}  
\begin{aligned}
			 [\Omega,\Pi_n] &= \Omega\Pi_n-\Omega^{[n]} \\&=\left[\begin{NiceMatrix}[columns-width=1.2cm]
			0 & \Cdots & 0 & \Cdots \\
			\Vdots & & \Vdots \\
			0 & \Cdots & \Omega_{n,n-M} & \Cdots & & \Omega_{n,n-1} & 0 & \Cdots \\
			\Vdots & & 0 & \Ddots &  & \Vdots & \Vdots &\\
			& & \Vdots & &  & &&\\
			& & & & \Ddots[shorten-end=20pt]& \Omega_{n+M-1,n-1} &&\\\\
			0 & \Cdots & 0 & \Cdots & & 0 &&\\
			\Vdots[shorten-end=2pt]& & \Vdots[shorten-end=2pt]& &  & \Vdots[shorten-end=2pt]
		\end{NiceMatrix}\right],
\end{aligned}
	\end{equation} 
	were $ \Omega_{n,n-M}$ sites in the $(n+1,n-M+1)$ entry of the represented matrix. For $n < M$ we assume that the commutator starts at the first column and the negative components should be ignored.
	
\begin{Proposition}
		The connection formulas for the kernel polynomials are: 
	\begin{multline}
		\label{KernerlConexMP} R(x)\hat{K}^{[n]}(x,y) = K^{[n]}(x,y)  + 
		\begin{bNiceMatrix}
			A_n^{(1)}(x) & \Cdots & A_{n+M-1}^{(1)}(x) \\
			\Vdots &  &  \Vdots[shorten-end=-4pt] \\
			A_n^{(p)}(x) & \Cdots & A_{n+M-1}^{(p)}(x)
		\end{bNiceMatrix}
		\\\times \begin{bNiceMatrix}
			\Omega_{n,n-1} & \Cdots & &\Omega_{n,n-M} \\
		&   &&0 \\
			\Vdots & \Iddots[shorten-end=-4pt]&\Iddots[shorten-end=10pt]&\Vdots\\
			\Omega_{n+M-1,n-1}&0&\Cdots& 0
		\end{bNiceMatrix} 
			 \begin{bNiceMatrix}
			\hat{B}^{(1)}_{n-1}(y) &  \Cdots & \hat{B}^{(q)}_{n-1}(y)  \\
			\Vdots & & \Vdots \\
			\hat{B}^{(1)}_{n-M} (y) &  \Cdots & \hat{B}^{(q)}_{n-M}(y) \\
		\end{bNiceMatrix}.
	\end{multline}
    As in equation \eqref{TruncaciónOmegaMP}, the case $n < M$ needs extra considerations. We will assume that $\Omega_{j,n-M+i} = 0$ and $\hat{B}_{n-M+i}^{(b)}=0$ when $n-M + i < 0$ for any $i$ and $j$. Another way to express this is by noting that, if $n<M$, the matrix of the $\Omega_{i,j}$ will be $M \times n$ and the $\hat{B}_i^{(b)}(y)$ components will range from $n-1$ to $0$.
\end{Proposition}
	\begin{proof}
		From the definition of the kernel polynomial, we have:
		\begin{align*}
			R(x)\hat{K}^{[n]}(x,y) &= R(x)\hat{A}(x)\Pi_n \hat{B}(y) = A(x)\Omega \Pi_n \hat{B}(y) = A(x) \Pi_n \Omega \hat{B}(y) + A(x)[\Omega, \Pi_n]\hat{B}(y)\\& = 
			K^{[n]}(x,y) + A(x)[\Omega, \Pi_n]\hat{B}(y).
		\end{align*}
		Recalling \eqref{TruncaciónOmegaMP}, for the last term in the above expression we find: 
		\begin{align*}
			& A(x)\begin{bNiceMatrix}
				0 & \Cdots & 0\\
				\Vdots & & \Vdots \\
				0 & \Cdots & 0 \\
				\sum_{i=1}^M\Omega_{n,n-i}\hat{B}^{(1)}_{n-i}(y) & \Cdots & \sum_{i=1}^M\Omega_{n,n-i}\hat{B}^{(q)}_{n-i}(y)\\ \\
				\sum_{i=1}^{M-1}\Omega_{n+1,n-i}\hat{B}^{(1)}_{n-i}(y) & \Cdots & \sum_{i=1}^{M-1}\Omega_{n+1,n-i}\hat{B}^{(q)}_{n-i}(y) \\
				\Vdots & & \Vdots \\
				\Omega_{n+M-2,n-2}\hat{B}^{(1)}_{n-2}(y) + \Omega_{n+M-2,n-1}\hat{B}^{(1)}_{n-1}(y) & \Cdots & \Omega_{n+M-2,n-2}\hat{B}^{(q)}_{n-2}(y) + \Omega_{n+M-2,n-1}\hat{B}^{(q)}_{n-1}(y) \\ \\
				\Omega_{n+M-1,n-1}\hat{B}^{(1)}_{n-1}(y) & \Cdots & \Omega_{n+M-1,n-1}\hat{B}^{(q)}_{n-1}(y)\\
				0 & \Cdots & 0\\
				\Vdots & & \Vdots
			\end{bNiceMatrix} \\ 
			& = \begin{bNiceMatrix}
				\sum_{j=0}^{M-1}\sum_{i=1}^{M-j}A^{(1)}_{n+j}(x)\Omega_{n+j,n-i}\hat{B}^{(1)}_{n-i}(y) & \Cdots & \sum_{j=0}^{M-1}\sum_{i=1}^{M-j}A^{(1)}_{n+j}(x)\Omega_{n+j,n-i}\hat{B}^{(q)}_{n-i}(y) \\[3pt]
				\sum_{j=0}^{M-1}\sum_{i=1}^{M-j}A^{(2)}_{n+j}(x)\Omega_{n+j,n-i}\hat{B}^{(1)}_{n-i}(y) & \Cdots & \sum_{j=0}^{M-1}\sum_{i=1}^{M-j}A^{(2)}_{n+j}(x)\Omega_{n+j,n-i}\hat{B}^{(q)}_{n-i}(y) \\
				\Vdots & & \Vdots \\
				\sum_{j=0}^{M-1}\sum_{i=1}^{M-j}A^{(p)}_{n+j}(x)\Omega_{n+j,n-i}\hat{B}^{(1)}_{n-i}(y) & \Cdots & \sum_{j=0}^{M-1}\sum_{i=1}^{M-j}A^{(p)}_{n+j}(x)\Omega_{n+j,n-i}\hat{B}^{(q)}_{n-i}(y) 
			\end{bNiceMatrix}  \\
			& = \begin{bNiceMatrix}
				A_n^{(1)}(x)  & \Cdots & A_{n+M-1}^{(1)}(x) \\
				\Vdots & & \Vdots \\
				A_n^{(p)}(x) & \Cdots & A_{n+M-1}^{(p)}(x)
			\end{bNiceMatrix} \begin{bNiceMatrix}
			\Omega_{n,n-1} & \Cdots & &\Omega_{n,n-M} \\
			&   &&0 \\
			\Vdots & \Iddots[shorten-end=-4pt]&\Iddots[shorten-end=10pt]&\Vdots\\
			\Omega_{n+M-1,n-1}&0&\Cdots& 0
			\end{bNiceMatrix} \begin{bNiceMatrix}
				\hat{B}^{(1)}_{n-1}(y)  & \Cdots & \hat{B}^{(q)}_{n-1}(y)  \\
				\Vdots & & \Vdots \\
				\hat{B}^{(1)}_{n-M} (y)  & \Cdots & \hat{B}^{(q)}_{n-M}(y) \\
			\end{bNiceMatrix}.
		\end{align*}
	\end{proof}
\begin{Definition}[$\tau$ determinants]
For $i\in\{0,\dots,M-1\}$, we introduce the  $\tau$ determinants defined as follows:
	\begin{align*}
		\tau_n &\coloneq  \begin{vNiceMatrix}
			\mathbb{A}_n (x_1) & \mathbb{A}_{n+1} (x_1) & \Cdots & \mathbb{A}_{n+M-1}(x_1) \\
			\mathbb{A}_n (x_2) & \mathbb{A}_{n+1} (x_2) & \Cdots & \mathbb{A}_{n+M-1}(x_2) \\
			\Vdots & \Vdots & \Ddots & \Vdots \\
			\mathbb{A}_n (x_M) & \mathbb{A}_{n+1} (x_M) & \Cdots & \mathbb{A}_{n+M-1}(x_M)    
		\end{vNiceMatrix} \\ \\
		\tau_n^{(i)}  &\coloneq   \begin{vNiceMatrix}
			\mathbb{A}_n (x_1) & \mathbb{A}_{n+1} (x_1) & \Cdots & \mathbb{A}_{n+i-1}(x_1) & \mathbb{A}_{n+i+1}(x_M) &\Cdots & \mathbb{A}_{n+M}(x_1) \\
			\mathbb{A}_n (x_2) & \mathbb{A}_{n+1} (x_2) & \Cdots & \mathbb{A}_{n+i-1}(x_2) & \mathbb{A}_{n+i+1}(x_M) &\Cdots & \mathbb{A}_{n+M}(x_2) \\
			\Vdots & \Vdots & & \Vdots & \Vdots & & \Vdots \\
			\mathbb{A}_n (x_M) & \mathbb{A}_{n+1} (x_M) & \Cdots & \mathbb{A}_{n+i-1}(x_M) & \mathbb{A}_{n+i+1}(x_M) &\Cdots & \mathbb{A}_{n+M}(x_M)
		\end{vNiceMatrix}.
	\end{align*}
\end{Definition}

\begin{Remark}
		Both definitions would coincide for the case $i=M-1$, $\tau_n = \tau_n^{(M-1)} $.
\end{Remark}

In terms of these $\tau$ determinants we find the following Christoffel formulas:
	
\begin{Theorem}[Christoffel formulas]\label{Theorem:Christoffel_Formulas}
		The following equalities hold:
	\begin{align}
		\label{Omegan+MMP} \Omega_{n+M,n} & = (-1)^{M}\frac{\tau_{n}}{\tau_{n+1}}, \: \: \: \Omega_{n+1+i,n}=(-1)^{i+1}\frac{\tau_n^{(i)}}{\tau_{n+1}},\\
		\label{ConexDefinitivaBMP} \hat{B}^{(b)}_{n-1}(y) & = \frac{\begin{vNiceMatrix}
				\mathbb{K}_b^{[n]}(x_1,y) & \mathbb{A}_n (x_1) & \Cdots & \mathbb{A}_{n+M-2}(x_1)  \\[3pt]
				\mathbb{K}_b^{[n]}(x_2,y) & \mathbb{A}_n (x_2) &  \Cdots & \mathbb{A}_{n+M-2}(x_2) \\
				\Vdots & \Vdots & & \Vdots \\
				\mathbb{K}_b^{[n]}(x_M,y) & \mathbb{A}_n (x_M) & \Cdots & \mathbb{A}_{n+M-2}(x_M) 
		\end{vNiceMatrix}}{\tau_{n-1}} ,\\ 
		\\
		\label{ConexDefinitivaAMP}  \left[R(x)\hat{A}(x)\right]_{n+1,a}  & = \frac{\begin{vNiceMatrix}
				A_n^{(a)}(x) & \Cdots & A_{n+M-1}^{(a)}(x) & A_{n+M}^{(a)}(x) \\
				\mathbb{A}_n(x_1) & \Cdots & \mathbb{A}_{n+M-1}(x_1) & \mathbb{A}_{n+M}(x_1) \\
				\Vdots & & \Vdots & \Vdots \\
				\mathbb{A}_n(x_M) & \Cdots &  \mathbb{A}_{n+M-1}(x_M) & \mathbb{A}_{n+M}(x_M)
		\end{vNiceMatrix}}{\tau_{n+1}} .
	\end{align}
\end{Theorem}
	\begin{proof}
		By left-multiplying the vector $\begin{bNiceMatrix} 
			0 & \Cdots & 0 & 1 
		\end{bNiceMatrix}$ to Equation \eqref{SistemaQMP}, we can solve for the entry $\Omega_{n+M-1,n-1}$
		\begin{align*}
			\Omega_{n+M-1,n-1} &  = - \begin{bNiceMatrix} 
				0 & \Cdots & 0 & 1 
			\end{bNiceMatrix}\begin{bNiceMatrix}
				\mathbb{A}_n(x_1) & \Cdots & \mathbb{A}_{n+M-1}(x_1) \\
				\mathbb{A}_n(x_2) & \Cdots & \mathbb{A}_{n+M-1}(x_2) \\
				\Vdots &  & \Vdots \\
				\mathbb{A}_n(x_M) & \Cdots & \mathbb{A}_{n+M-1}(x_M) 
			\end{bNiceMatrix}^{-1} \begin{bNiceMatrix}
				\mathbb{A}_{n-1}(x_1) \\
				\mathbb{A}_{n-1}(x_2) \\
				\Vdots \\
				\mathbb{A}_{n-1}(x_M)
			\end{bNiceMatrix}  \\
			& = \frac{\begin{vNiceMatrix}
					\mathbb{A}_n(x_1) & \Cdots & \mathbb{A}_{n+M-1}(x_1) & \mathbb{A}_{n-1}(x_1) \\
					\Vdots &  & \Vdots & \Vdots \\
					\mathbb{A}_n(x_M) & \cdots & \mathbb{A}_{n+M-1}(x_M) & \mathbb{A}_{n-1}(x_M) \\
					0 & \Cdots & 1 & 0
			\end{vNiceMatrix}}{\tau_n} \\&= - \frac{\begin{vNiceMatrix}
					\mathbb{A}_n(x_1) & \Cdots & \mathbb{A}_{n+M-2}(x_1) & \mathbb{A}_{n-1}(x_1) \\
					\Vdots &  & \Vdots & \Vdots \\
					\mathbb{A}_n(x_M) & \Cdots & \mathbb{A}_{n+M-2}(x_M) & \mathbb{A}_{n-1}(x_M) 
			\end{vNiceMatrix}}{\tau_n} \\
			& = (-1)^{M}\frac{\tau_{n-1}}{\tau_n}.
		\end{align*}
		The entry $\Omega_{n+i,n-1}$ is obtained by left-multiplying by the vector 
		\begin{equation*}
			\begin{bNiceMatrix}
				0 & \Cdots & (\text{i-times}) & \Cdots & 0 & 1 & 0 & \Cdots & 0
			\end{bNiceMatrix}, 
		\end{equation*}
		and after an analogous procedure, we arrive at both relationships appearing in Equation \eqref{Omegan+MMP}.

		We introduce the following notation:
		 \begin{gather*}
			\left[ \Omega_M \right]  \coloneq \begin{bNiceMatrix}
				\Omega_{n,n-1} & \Cdots & &\Omega_{n,n-M} \\
				&   &&0 \\
				\Vdots & \Iddots[shorten-end=-4pt]&\Iddots[shorten-end=10pt]&\Vdots\\
				\Omega_{n+M-1,n-1}&0&\Cdots& 0
			\end{bNiceMatrix},\\
			\begin{aligned}
				\left[ \hat{B}(y) \right] & = \begin{bNiceMatrix}
					\hat{B}^{(1)}_{n-1}(y) &  \Cdots & \hat{B}^{(q)}_{n-1}(y)  \\
					\Vdots & & \Vdots \\
					\hat{B}^{(1)}_{n-M} (y) &  \Cdots & \hat{B}^{(q)}_{n-M}(y) \\
				\end{bNiceMatrix}, & \left[ \hat{B}(y) \right]_b & = \begin{bNiceMatrix}
					\hat{B}^{(b)}_{n-1}(y) \\
					\Vdots \\
					\hat{B}^{(b)}_{n-M}(y)
				\end{bNiceMatrix}.
			\end{aligned}
		\end{gather*}	
		From Equation \eqref{KernerlConexMP}, we can evaluate at $x=x_1$ and left multiply  with the vector $\boldsymbol{v}_1$. We obtain:
		\begin{equation*}
			0 = \boldsymbol{v}_1K^{[n]}(x_1,y) + \boldsymbol{v}_1 \begin{bNiceMatrix}
				A_n^{(1)}(x) & A_{n+1}^{(1)}(x) & \Cdots & A_{n+M-1}^{(1)}(x) \\[3pt]
				A_n^{(2)}(x) & A_{n+1}^{(2)}(x) & \Cdots & A_{n+M-1}^{(2)}(x) \\
				\Vdots & \Vdots & & \Vdots \\
				A_n^{(p)}(x) & A_{n+1}^{(p)}(x) & \Cdots & A_{n+M-1}^{(p)}(x)
			\end{bNiceMatrix} \begin{bmatrix}
				\Omega_{M}
			\end{bmatrix} \begin{bmatrix}
				\hat{B}(y)
			\end{bmatrix} . 
		\end{equation*}
	The last equation, when dealing only with the \(b\)-th column of the matrix expression, implies: 
		\begin{equation*}
			-\mathbb{K}_b^{[n]}(x_1,y) = \begin{bNiceMatrix}
				\mathbb{A}_n(x_1) & \mathbb{A}_{n+1}(x_1) & \Cdots & \mathbb{A}_{n+M-1}(x_1)
			\end{bNiceMatrix}\begin{bmatrix}
				\Omega_{M}
			\end{bmatrix} \begin{bmatrix}
				\hat{B}(y)
			\end{bmatrix}_b .
		\end{equation*}
		Taking into account the system that appears for the roots and their corresponding eigenvectors, we obtain:
		\begin{multline*}
			-\begin{bNiceMatrix}
				\mathbb{K}_b^{[n]}(x_1,y) \\[3pt]
				\mathbb{K}_b^{[n]}(x_2,y) \\
				\Vdots \\
				\mathbb{K}_b^{[n]}(x_M,y)
			\end{bNiceMatrix}   = \begin{bNiceMatrix}
				\mathbb{A}_n(x_1) & \mathbb{A}_{n+1}(x_1) & \Cdots & \mathbb{A}_{n+M-1}(x_1) \\[3pt]
				\mathbb{A}_n(x_2) & \mathbb{A}_{n+1}(x_2) & \Cdots & \mathbb{A}_{n+M-1}(x_2) \\
				\Vdots & \Vdots & & \Vdots \\
				\mathbb{A}_n(x_M) & \mathbb{A}_{n+1}(x_M) & \Cdots & \mathbb{A}_{n+M-1}(x_M) \\
			\end{bNiceMatrix}\begin{bmatrix}
				\Omega_{M}
			\end{bmatrix} \begin{bmatrix}
				\hat{B}(y)
			\end{bmatrix}_b \\
			-\begin{bNiceMatrix}
				\mathbb{A}_n(x_1) & \Cdots & \mathbb{A}_{n+M-1}(x_1) \\
				\Vdots & & \Vdots \\
				\mathbb{A}_n(x_M) &\Cdots & \mathbb{A}_{n+M-1}(x_M) \\
			\end{bNiceMatrix}^{-1}   \begin{bNiceMatrix}
				\mathbb{K}_b^{[n]}(x_1,y) \\[3pt]
				\mathbb{K}_b^{[n]}(x_2,y) \\
				\Vdots \\
				\mathbb{K}_b^{[n]}(x_M,y)
			\end{bNiceMatrix}  = \begin{bmatrix}
				\Omega_{M}
			\end{bmatrix} \begin{bmatrix}
				\hat{B}(y)
			\end{bmatrix}_b .
		\end{multline*}
		By left-multiplying the vector $\begin{bNiceMatrix} 0 & \Cdots & 0 & 1 \end{bNiceMatrix}$ to both members of the last expression:
		\begin{gather*}
			\begin{multlined}[t][.9\textwidth]
				-\begin{bNiceMatrix} 0 & \Cdots & 0 & 1 \end{bNiceMatrix} \begin{bNiceMatrix}
				\mathbb{A}_n(x_1) & \Cdots & \mathbb{A}_{n+M-1}(x_1) \\
				\Vdots & & \Vdots \\
				\mathbb{A}_n(x_M) &\Cdots & \mathbb{A}_{n+M-1}(x_M) \\
			\end{bNiceMatrix}^{-1} \begin{bNiceMatrix}
				\mathbb{K}_b^{[n]}(x_1,y) \\[3pt]
				\mathbb{K}_b^{[n]}(x_2,y) \\
				\Vdots \\
				\mathbb{K}_b^{[n]}(x_M,y)
			\end{bNiceMatrix}  \\
			= -\frac{1}{\tau_n}\begin{vNiceMatrix}
				\mathbb{A}_n(x_1) & \Cdots & \mathbb{A}_{n+M-2}(x_1) & \mathbb{K}_b^{[n]}(x_1,y) \\[3pt]
				\mathbb{A}_n(x_2) & \Cdots & \mathbb{A}_{n+M-2}(x_2) & \mathbb{K}_b^{[n]}(x_2,y) \\
				\Vdots & & \Vdots & \Vdots \\
				\mathbb{A}_n(x_M) & \Cdots &  \mathbb{A}_{n+M-2}(x_M) & \mathbb{K}_b^{[n]}(x_M,y) 
			\end{vNiceMatrix}, 
			\end{multlined}\\
		\begin{multlined}[t][.9\textwidth]
				\begin{bNiceMatrix} 0 & \Cdots & 0 & 1 \end{bNiceMatrix}\begin{bmatrix}
				\Omega_{M}
			\end{bmatrix} \begin{bmatrix}
				\hat{B}(y)
			\end{bmatrix} \\
			= \begin{bNiceMatrix} 0 & \Cdots & 0 & 1 \end{bNiceMatrix}\begin{bNiceMatrix}
				\Omega_{n,n-1} & \Cdots & &\Omega_{n,n-M} \\
				&   &&0 \\
				\Vdots & \Iddots[shorten-end=-4pt]&\Iddots[shorten-end=10pt]&\Vdots\\
				\Omega_{n+M-1,n-1}&0&\Cdots& 0
				\end{bNiceMatrix}
			\begin{bNiceMatrix}
				\hat{B}_{n-1}^{(b)}(y) \\
				\Vdots \\
				\hat{B}_{n-M+1}^{(b)}(y) \\[3pt]
				\hat{B}_{n-M}^{(b)} (y)
			\end{bNiceMatrix}  \\
			= \Omega_{n+M-1,n-1}\hat{B}^{(b)}_{n-1}(y).
		\end{multlined}
		\end{gather*}
		The last expression in combination with Equation \eqref{Omegan+MMP} leads to Equation \eqref{ConexDefinitivaBMP}.
		
		We can expand Equation \eqref{FormulasConexAMP} into entries to obtain:
		\begin{align*}
			\left[R(x)\hat{A}(x)\right]_{n+1,a} & = \sum_{\Bar{a}=1}^p \left(R(x)\right)_{a,\Bar{a}}\hat{A}_n^{(\Bar{a})}(x) = A_n^{(a)}(x) + \begin{bNiceMatrix}
				A_{n+1}^{(a)}(x) & \Cdots & A_{n+M-1}^{(a)}(x) & A_{n+M}^{(a)}(x)
			\end{bNiceMatrix} \begin{bNiceMatrix}
				\Omega_{n+1,n} \\
				\Vdots \\
				\Omega_{n+M-1,n} \\
				\Omega_{n+M,n}
			\end{bNiceMatrix} \\
			& = A_n^{(a)}-\begin{bNiceMatrix}
				A_{n+1}^{(a)}(x) & \Cdots & A_{n+M}^{(a)}(x)
			\end{bNiceMatrix}\begin{bNiceMatrix}
				\mathbb{A}_{n+1}(x_1) & \Cdots & \mathbb{A}_{n+M}(x_1) \\
				\Vdots &  & \Vdots \\
				\mathbb{A}_{n+1}(x_M) & \Cdots & \mathbb{A}_{n+M}(x_M) 
			\end{bNiceMatrix}^{-1} \begin{bNiceMatrix}
				\mathbb{A}_{n}(x_1) \\
				\Vdots \\
				\mathbb{A}_{n}(x_M)
			\end{bNiceMatrix} \\
			& = \frac{1}{\tau_{n+1}}\begin{vNiceMatrix}
				\mathbb{A}_{n+1}(x_1) & \Cdots & \mathbb{A}_{n+M}(x_1) & \mathbb{A}_{n}(x_1) \\
				\Vdots &  & \Vdots & \Vdots \\
				\mathbb{A}_{n+1}(x_M) & \Cdots & \mathbb{A}_{n+M}(x_M) & \mathbb{A}_{n}(x_1) \\
				A_{n+1}^{(a)}(x) & \Cdots & A_{n+M}^{(a)}(x) & A_{n}^{(a)}(x)
			\end{vNiceMatrix},
		\end{align*}
		rearranging rows and columns yields the mentioned result.
	\end{proof}
\begin{Remark}
    As will be shown in section \S \ref{S:Criteria}, a construction of the form $R(x)\hat{A}(x)=A(x)\Omega$ is divisible by $R(x)$. Although no explicit formula can be obtained for the adjugate matrix of $R(x)$, in the general case, the following decomposition will always be possible:
    \begin{equation*}
     \hat{A}_n^{(a)}(x)  = \frac{1}{\det R(x)\, \tau_n} \sum_{\Bar{a}=1}^p \left( \mathrm{adj} \, R(x) \right)_{a,\Bar{a}}\begin{vNiceMatrix}
				A_n^{(\Bar{a})}(x) & \Cdots & A_{n+M-1}^{(\Bar{a})}(x) & A_{n+M}^{(\Bar{a})}(x) \\
				\mathbb{A}_n(x_1) & \Cdots & \mathbb{A}_{n+M-1}(x_1) & \mathbb{A}_{n+M}(x_1) \\
				\Vdots & & \Vdots & \Vdots \\
				\mathbb{A}_n(x_M) & \Cdots &  \mathbb{A}_{n+M-1}(x_M) & \mathbb{A}_{n+M}(x_M)
		\end{vNiceMatrix}
    \end{equation*}
\end{Remark}
	\subsection{Eigenvalues with Arbitrary Multiplicity} \label{S:ArbitraryMultiplicity}
		Previously, it was imposed that the zeros of $R(x)$ were simple. However, this is a too restrictive condition; for any matrix polynomial as in condition (\ref{CondicionesMatricesLideresFinales}), the proposed method can be applied. 
	We recover the notation used in Proposition \ref{SmithForm}, that is, $x_i$, with $i\in\{1,\cdots, M\}$, denote the distinct zeros of $\det R(x)$, $\kappa_{i,j}$ represents partial multiplicities, and $K_i$ denotes the multiplicity of $x_i$ as a zero of $\det R(x)$, then:
	\[ 
	\begin{aligned}
		K_i &= \sum_{j=1}^{p} \kappa_{i,j}, & \sum_{i=1}^{M}\sum_{j=1}^p \kappa_{i,j} &= Np - r,
	\end{aligned}
	\]
	and the partial multiplicities may assume zero values in certain scenarios. In the first section, it was already explained how, from a zero of the determinant of $R(x)$, $x_i$, we can construct $s_i$ Jordan chains (with $s_i \leq p$) of length $\kappa_{i,j}$. To generalize the result, we will simply see that:
	
	\begin{Proposition}
		If $\{\boldsymbol{v}_{i,0},\: \boldsymbol{v}_{i,1}, \cdots, \boldsymbol{v}_{i,\kappa_i-1} \}$ is a Jordan chain of $R(x)$ corresponding to the eigenvalue $\lambda$, it will also be so for $\mathcal A_n(x)\Omega$. 
	\end{Proposition}
	
	\begin{proof}
		From Equation \eqref{FormulasConexAMP}, we can differentiate it $l$ times: 
		\begin{align*}
			\frac{\d^l A(x)}{\d x^l}\Omega = \frac{\d^l}{\d x^l}\left[ R(x) \widetilde{A}(x)\right] = \sum_{j=0}^l \binom{l}{j}R^{(l-j)}(x)\widetilde{A}^{(j)}(x), 
		\end{align*}
		where now the superscript $(j)$ refers to the $j$-th derivative (previously, this superscript was used in another context). We can evaluate the expression at $\lambda$, act from the left by the corresponding eigenvector $\boldsymbol{v}_{\kappa_{i}-l}$ and sum over $l$:
		\begin{align*}
			\sum_{l=0}^L \frac{1}{l!} \boldsymbol{v}_{\kappa_{i}-1-l} A^{(l)}(\lambda)\Omega  =& \sum_{l=0}^L \frac{1}{l!}\sum_{j=0}^l \binom{l}{j} \boldsymbol{v}_{\kappa_{i}-1-l}R^{(l-j)}(\lambda)\widetilde{A}^{(j)}(\lambda) \\ 
			=& \frac{1}{L!}\binom{L}{L}\boldsymbol{v}_0R^{(0)}(\lambda)\widetilde{A}^{(L)}(\lambda) + \\
			& + \left[ \frac{1}{L!}\binom{L}{L-1}\boldsymbol{v}_0R^{(1)}(\lambda)+\frac{1}{(L-1)!}\binom{L-1}{L-1}\boldsymbol{v}_1R^{(0)}(\lambda) \right]\widetilde{A}^{(L-1)}(\lambda) \\
			& \; \vdots \\
			& + \left[ \sum_{l=0}^L \frac{1}{l!} \binom{l}{0} \boldsymbol{v}_{\kappa_i-1-l}R^{(l)}(\lambda)\right]\widetilde{A}^{(0)}(\lambda) \\
			= &\frac{1}{L!} \left[ \boldsymbol{v}_0R^{(0)}(\lambda)\right]\widetilde{A}^{(L)}(\lambda) + \\
			& + \frac{1}{(L-1)!} \left[ \boldsymbol{v}_0R^{(1)}(\lambda)+\boldsymbol{v}_1R^{(0)}(\lambda)\right]\widetilde{A}^{(L-1)}(\lambda)  \\
			& \;\vdots \\
			& + \left[ \sum_{l=0}^L \frac{1}{l!}\boldsymbol{v}_{\kappa_i-1-l}R^{(l)}(\lambda)\right]\widetilde{A}^{(0)}(\lambda) = 0 . 
		\end{align*}
		Each term separately vanishes, and the result is proven for $L\in\{0,\cdots \kappa_i-1\}$.
	\end{proof}
	Associated with a zero of the determinant, $x_i$, we have a canonical set of Jordan chains of length $K_i$: 
	\[ 
	\{ \boldsymbol{v}_{1,0},\boldsymbol{v}_{1,1}, \dots, \boldsymbol{v}_{1,\kappa_1-1}, \boldsymbol{v}_{2,0}, \boldsymbol{v}_{2,1}, \dots,\boldsymbol{v}_{2,\kappa_2-1}, \dots,\: \boldsymbol{v}_{s,0},\boldsymbol{v}_{s,1}, \cdots, \boldsymbol{v}_{s,\kappa_s-1} \}, 
	\]
	where we have already demonstrated that each Jordan chain of length $\kappa_{i,j}-1$ for $j\in\{1,\cdots,s\}$ of the matrix $R(x)$ also holds for $A(x)\Omega$. We can now construct vectors of length $\kappa_{i,j}$ as follows: 
	\begin{equation*}
		\mathbb{A}_{n;j}(x_i) = \begin{bNiceMatrix}
		\displaystyle	\sum_{a=1}^p \boldsymbol{v}_{j,0;a}A_n^{(a)}(x_i) & \Cdots & \displaystyle\sum_{l=0}^{\kappa_{i,j}-1}\dfrac{1}{l!}\displaystyle\sum_{a=1}^p \boldsymbol{v}_{1,\kappa_{i,j}-1-l;a}\left.\dfrac{\d^l A_n^{(a)}}{\d x^l}\right|_{{x}=x_i}
		\end{bNiceMatrix}^\top.
	\end{equation*}
	From which we define a vector of length $K_i$: 
	\begin{equation*}
		\mathbb{A}_{n}(x_i) = \begin{bNiceMatrix}
			\mathbb{A}_{n;1}(x_i) & \mathbb{A}_{n;2}(x_i) & \Cdots & \mathbb{A}_{n;s}(x_i)
		\end{bNiceMatrix}^\top.
	\end{equation*}
	Recalling that
	\begin{equation*}
		\begin{bNiceMatrix}
			K_{1,b}^{[n]}(x,y) \\[3pt]
			K_{2,b}^{[n]}(x,y) \\
			\Vdots \\
			K_{p,b}^{[n]}(x,y) \\
		\end{bNiceMatrix} = \sum_{j=0}^{n-1}\begin{bNiceMatrix}
			A_j^{(1)}(x) \\[3pt]
			A_j^{(2)}(x) \\
			\Vdots \\
			A_j^{(p)}(x) \\
		\end{bNiceMatrix}B_j^{(b)}(y),
	\end{equation*}
 the generalization of the previous notation for this case is immediate (note that the $B_j^{(b)}(y)$ depend only on $y$ and thus are not affected when deriving with respect to $x$ in the construction of the Jordan chains). In terms of these vectors, the results presented throughout this article are expressed similarly (bearing in mind that now the determinants are $(Np-r )\times (Np-r)$ and not just $M \times M$). For example, the  $\tau$ determinants will be of the form: 
	\begin{equation*}
		\tau_n \coloneq  \begin{vNiceMatrix}
			\mathbb{A}_n (x_1) & \mathbb{A}_{n+1} (x_1) & \Cdots & \mathbb{A}_{n+Np-r-1}(x_1) \\
			\mathbb{A}_n (x_2) & \mathbb{A}_{n+1} (x_2) & \Cdots & \mathbb{A}_{n+Np-r-1}(x_2) \\
			\Vdots & \Vdots & \Ddots & \Vdots \\
			\mathbb{A}_n (x_M) & \mathbb{A}_{n+1} (x_M) & \Cdots & \mathbb{A}_{n+Np-r-1}(x_M)    
		\end{vNiceMatrix}.
	\end{equation*}
	\subsection{Left perturbations}
	Next, we will consider perturbations of the  matrix of  measures  by left multiplication with another matrix polynomial:
	\begin{equation} \label{multporizq}
		\d\hat{\mu}(x) = L(x)\d\mu(x).
	\end{equation}
	It will not be necessary to repeat the same arguments as before to ensure the existence of perturbed orthogonality (with the condition $\tau_n \neq 0$), since everything done so far is easily generalizable to this other case. The moment matrix for the family of initial polynomials had a Gauss-Borel factorization, given by:
	\begin{equation*}
		\mathcal{M} = \int X_{[q]}(x) \d\mu (x) X_{[p]}^\top (x) \: dx = S^{-1}H\Bar{S}^{-\top},
	\end{equation*}
	from which the matrix polynomials could be defined:
	\begin{align*}
		A(x) & = X_{[p]}^\top (x) \Bar{S}^\top H^{-1}, \\
		B(x) & = S X_{[q]}(x).
	\end{align*}
	The biorthogonality relation for these matrix polynomials is as follows:
	\begin{equation*}
		\int B(x) \d\mu(x)A(x) = I.
	\end{equation*}
	We can now consider a new moment matrix:
	\begin{equation*}
		\mathcal{M}^\top = \mathcal{M}' \: \Rightarrow \: \d\mu^\top (x) = \d\mu'(x),
	\end{equation*}
	and redefine the matrices of the Gauss-Borel factorization,
	\begin{equation*}
		\Bar{S}=S', \: S = \Bar{S}' \: \Rightarrow \: A'(x) = X_{[p']}^\top (x) \Bar{S}' H^{-1}, \: B'(x) = S'X_{[q']}(x),
	\end{equation*}
	where $p' = q$ and $q' = p$. Nothing changes if we transpose the roles played by the polynomials $A(x)$ and $B(x)$ (by properly transposing the weight matrix). Therefore, if we consider Equation \eqref{multporizq} as follows:
	\begin{equation*}
		\d\hat{\mu}(x) = L(x) \d\mu'(x) = \d\mu^\top(x) R(x),
	\end{equation*}
	the orthogonality for this perturbed weight matrix is demonstrated if $R(x) = L^\top (x)$ is a matrix polynomial as studied in conditions \eqref{CondicionesMatricesLideresFinales} and $\tau_n \neq 0$.
	
	For convenience, we will give explicit formulas and results for this case, as well as some guidelines for obtaining perturbed polynomials without having to transpose the relations. Let's start by studying the matrix polynomials $L(x)$; we will consider matrices of the form:
	\begin{align*}
\begin{aligned}
			L(x) & = \sum_{l=0}^N L_l x^l , & L_l \in \mathbb{C}^{p \times p},
\end{aligned}
	\end{align*}
	whose leading matrices are
	\begin{align*}
		L_{N}  & = \begin{bmatrix}
			0_{r\times (q-r)} & 0_{r} \\ \\
			I_{(q-r)} & 0_{(q-r) \times r}
		\end{bmatrix}, & L_{N-1}  & = \begin{bmatrix}
			\left[ L^1_{N-1} \right]_{r\times (q-r)} & I_{r} \\ \\
			\left[ L^3_{N-1} \right]_{(q-r)} & \left[ L^4_{N-1} \right]_{(q-r) \times r}
		\end{bmatrix}.
	\end{align*}
	For this case, the determinant of the matrix $L(x)$ is of degree $qN-r$. The relations between the moment matrices are given by:
	\begin{equation*}
		\hat{\mathcal{M}} = L\left( \Lambda_{[q]} \right) \mathcal{M}. 
	\end{equation*}
\begin{Definition}
		We define a matrix $\Omega$, through $L\left( \Lambda_{[q]} \right)$ and the matrices of the Gauss--Borel factorization:
	\begin{equation*}
		\Omega\coloneq  \hat{S}L\left( \Lambda_{[q]} \right)S^{-1} = \hat{H}\hat{\Bar{S}}^{-\top}\Bar{S}^\top H^{-1},
	\end{equation*}
\end{Definition}
	where now the matrix $\Omega$ is upper triangular with $qN-r$ upper-diagonals. 
\begin{Proposition}
		The connection formulas between the initial and perturbed polynomials are given by:
\[	\begin{aligned}
		\Omega B(x) & = \hat{B}(x)L(x), & \hat{A}(x)\Omega & = A(x),
	\end{aligned}\]
	and between the kernel CD polynomials:
	\begin{multline*}
		 \hat{K}^{[n]}(x,y) L(y) = K^{[n]}(x,y) + 
	 \begin{bNiceMatrix}
			\hat{A}_{n-1}^{(1)}(x) & \Cdots & \hat{A}_{n-qN-r}^{(1)}(x) \\
			\Vdots &  & \Vdots \\
			\hat{A}_{n-1}^{(p)}(x) & \Cdots & \hat{A}_{n-qN-r}^{(p)}(x) \\
		\end{bNiceMatrix}\\\times
	\begin{bNiceMatrix}[columns-width=1.35cm]
		\Omega_{n-1,n} & \Cdots & &\Omega_{n-1,n+qN-r-1}  \\
		&   &&0 \\
		\Vdots & \Iddots[shorten-end=-4pt]&\Iddots[shorten-end=10pt]&\Vdots\\[60pt]
		\Omega_{n-qN-r,n}&0&\Cdots& 0
	\end{bNiceMatrix} \begin{bNiceMatrix}
			B_n^{(1)}(y) & \Cdots & B_n^{(q)}(y) \\
			\Vdots & & \Vdots \\
			B_{n+qN-r-1}^{(1)}(y) & \Cdots & B_{n+qN-r-1}^{(q)}(y)
		\end{bNiceMatrix}.
	\end{multline*}
\end{Proposition}

	Similarly to how left Jordan chains were initially studied for a given matrix polynomial, the study of right Jordan chains is completely analogous for this case. 
	\begin{Definition}
	Assuming simple eigenvalues, 	for left perturbations we introduce the following notation
	\begin{align*}
		\mathbb{B}_n(y_i) & = \left( B(y_i)\boldsymbol{v}_i \right)_{n+1} = \sum_{b=1}^qB_n^{(b)}(y_i)v_{i,b}, & \mathbb{K}_a^{[n]}(x,y_i) = \sum_{b=1}^qK_{a,b}^{[n]}(x,y_i)v_{i,b}, 
	\end{align*}
	\begin{equation*}
		\tau_n = \begin{vNiceMatrix}
			\mathbb{B}_n(y_1) & \Cdots & \mathbb{B}_n(y_M) \\
			\Vdots & & \Vdots \\
			\mathbb{B}_{n+M-1}(y_1) & \Cdots & \mathbb{B}_{n+M-1}(y_M)
		\end{vNiceMatrix},
	\end{equation*}
	where for simplicity we identify one eigenvector with its corresponding eigenvalue. 
	\end{Definition}
	More general cases can be easily generalized by introducing more notation. 
	
	With all this, 
	\begin{Theorem}[Christoffel formulas]
		The left perturbed matrix polynomials entries are:
	\begin{align*}
		\hat{A}^{(a)}_{n-1} & = \frac{1}{\tau_{n-1}}\begin{vNiceMatrix}
			\mathbb{K}_a^{[n]}(x,y_1) & \mathbb{K}_a^{[n]}(x,y_2) & \Cdots & \mathbb{K}_a^{[n]}(x,y_M) \\
			\mathbb{B}_n(y_1) & \mathbb{B}_n(y_2) & \Cdots & \mathbb{B}_n(y_M) \\
			\Vdots & \Vdots & & \Vdots \\
			\mathbb{B}_{n+M-1}(y_1) & \mathbb{B}_{n+M-1}(y_2) &\Cdots & \mathbb{B}_{n+M-1}(y_M)
		\end{vNiceMatrix}, \\ \\ \left[ \hat{B}(x)L(x) \right]_{n+1,b} & = \frac{1}{\tau_{n+1}}\begin{vNiceMatrix}
			B_n^{(b)}(x) & \mathbb{B}_n(y_1) & \Cdots & \mathbb{B}_n(y_M) \\[3pt]
			B_{n+1}^{(b)}(x) & \mathbb{B}_{n+1}(y_1) & \Cdots & \mathbb{B}_{n+1}(y_M) \\
			\Vdots & \Vdots & & \Vdots \\
			B_{n+M}^{(b)}(x) & \mathbb{B}_{n+M}(y_1) & \Cdots & \mathbb{B}_{n+M}(y_M)
		\end{vNiceMatrix}.
	\end{align*}
	\end{Theorem}

	\section{On the Existence of Christoffel Perturbed Orthogonality}\label{S:Criteria}
	In the previous discussion, we initially assumed the existence of the family of orthogonal polynomials. Now we give sufficient  and necessary conditions for this existence. 
	We will use the following block expression 
	\begin{align*}
		A(x)=\left[\begin{NiceMatrix}
			\mathcal A_0 &\mathcal A_1 (x)&\mathcal A_2(x)&\Cdots
			\end{NiceMatrix}\right],
	\end{align*}
	where $\mathcal A_n(x)$ is a $p\times p$ matrix polynomial of degree $n$, with leading coefficient an invertible upper triangular matrix. 
	Given the mixed multiple orthogonal polynomials $A(x)$  and the previously defined matrix  $\Omega$, let us denote
	\begin{equation}
		\label{ATildeMP}    
		\begin{aligned}
			A(x) \Omega &\eqcolon \widetilde{A}(x), & \mathcal A_n(x) \Omega &\eqcolon \widetilde{\mathcal A}_n(x), 
		\end{aligned}
	\end{equation}
with the matrix polynomials $\widetilde{\mathcal A}_n(x)$ the corresponding blocks, i.e,
		\begin{align*}
		\widetilde A(x)=
	\left[	\begin{NiceMatrix}
		\widetilde{\mathcal A}_0 &\widetilde{\mathcal A}_1 (x)&\widetilde{\mathcal A}_2(x)&\Cdots
		\end{NiceMatrix}\right].
	\end{align*}

	\begin{Proposition}
	The matrix polynomials	$\widetilde{\mathcal A}_n(x)$ has, at least, the same left eigenvectors associated with the same eigenvalue as $R(x)$. 
	\end{Proposition}
	
	\begin{proof} 
		Entrywise, $A(x)\Omega$ reads as follows: 
		\begin{align*}
			A_{n}^{(a)}(x) & + \Omega_{n+1,n}A_{n+1}^{(a)}(x) + \Omega_{n+2,n}A_{n+2}^{(a)}(x) + \cdots + \Omega_{n+M,n}A_{n+M}^{(a)}(x) = \widetilde{A}_n^{(a)}(x),
		\end{align*}
		that can be written as
		\begin{align*}
			A_{n}^{(a)}(x) & - \frac{\tau_n^{(0)}}{\tau_{n+1}}A_{n+1}^{(a)}(x) +\frac{\tau_n^{(1)}}{\tau_{n+1}}A_{n+2}^{(a)}(x) + \cdots + (-1)^{M}\frac{\tau_{n}}{\tau_{n+1}}A_{n+M}^{(a)}(x) = \widetilde{A}_n^{(a)}(x).
		\end{align*}
		When evaluating the expressions at $x=x_i$ and acting from the left by the corresponding eigenvector: 
		\begin{align*}
&	\begin{multlined}[t][.95\textwidth]
\sum_{a=1}^pv_{i,a}A_{n}^{(a)}(x_i) - \frac{\tau_n^{(M-1)}}{\tau_{n+1}}\sum_{a=1}^pv_{i,a}A_{n+1}^{(a)}(x_i) +\frac{\tau_n^{(M-2)}}{\tau_{n+1}}\sum_{a=1}^pv_{i,a}A_{n+2}^{(a)}(x_i) + \cdots \\+ (-1)^{M}\frac{\tau_{n}}{\tau_{n+1}}\sum_{a=1}^pv_{i,a}A_{n+M}^{(a)}(x_i)
	\end{multlined} \\
	&\begin{aligned}
				&= \sum_{a=1}^pv_{i,a}\widetilde{A}_n^{(a)}(x_i) 
			\tau_{n+1}\mathbb{A}_{n}(x_i)  - \tau_n^{(M-1)}\mathbb{A}_{n+1}(x_i) +\tau_n^{(M-2)}\mathbb{A}_{n+2}(x_i) + \cdots + (-1)^{M}\tau_{n}\mathbb{A}_{n+M}(x_i) \\&= \tau_{n+1}\widetilde{\mathbb{A}}_n(x_i)
	\end{aligned}
		\end{align*}
		The right-hand side is simply a determinant expanded by the Laplace method in the first column:
		\begin{equation*}
			\begin{vNiceMatrix}
				\mathbb{A}_n(x_i) & \mathbb{A}_{n+1}(x_i) & \mathbb{A}_{n+2}(x_i) & \Cdots & \mathbb{A}_{n+M}(x_i) \\[3pt]
				\mathbb{A}_n(x_1) & \mathbb{A}_{n+1}(x_1) & \mathbb{A}_{n+2}(x_1) & \Cdots & \mathbb{A}_{n+M}(x_1) \\
				\Vdots & \Vdots & \Vdots & & \Vdots \\
				\mathbb{A}_n(x_M) & \mathbb{A}_{n+1}(x_M) & \mathbb{A}_{n+2}(x_M) & \Cdots & \mathbb{A}_{n+M}(x_M)
			\end{vNiceMatrix} = \tau_{n+1}\widetilde{\mathbb{A}}_n(x_i).
		\end{equation*}
		This is equal to zero, since $x_i$ can be any of the $M$ roots of $R(x)$ and thus one row will always repeat. With all this, we have proven the result.
	\end{proof}
	\begin{Theorem}\label{Theorem: Criteria}
		The perturbed mixed multiple orthogonality exists if and only if  $\tau_n\neq 0$ for $n\in\N_0$.
		 \end{Theorem}
	\begin{proof}
First we deal with the sufficiency. Since $\widetilde{A}(x)$ and $R(x)$ share Jordan chains for each eigenvalue, according to Theorem \ref{Theorem}, $R(x)$ is a divisor of $\widetilde{A}(x)$. The matrix polynomial $\widetilde{A}(x)$ is a semi-infinite matrix with an infinite number of $p \times p$ blocks, so the divisibility should be understood between $R(x)$ and each one of these block matrices respectively. There exists another matrix polynomial $\widetilde{\widetilde{A\,}}$ that satisfies  
	\begin{equation*}
		A(x)\Omega = \widetilde{A}(x) = R(x)\widetilde{\widetilde{A\,}}(x) 
	\end{equation*}
	In terms of $p\times p$ matrices, the first block can be written as
	\begin{align*}
		\mathcal{A}_0(x)\left[\Omega_{0,0}\right]+\mathcal{A}_p(x)\left[\Omega_{p,0}\right] + \cdots + \mathcal{A}_{pN-p}(x)\left[\Omega_{pN-p,0}\right]+\mathcal{A}_{pN}(x)\left[\Omega_{pN,0}\right] = R(x) \widetilde{\widetilde{\mathcal{A\,}}}_0(x).
	\end{align*}
	By construction, the matrix $\left[ \Omega_{pN,0} \right]$ and the matrix $\left[ \Omega_{pN-p,0} \right]$ are like the leading and subleading matrices exposed in condition \eqref{LeadingMatrixConditions}, respectively. Moreover, $\mathcal{A}_{pN}(x)$ can be understood as a matrix polynomial of degree $N$ whose leading matrix is upper triangular with a nonzero determinant. The product $\mathcal{A}_{pN}(x)\left[ \Omega_{pN,0} \right]$ will be another matrix polynomial that satisfies, 
	\begin{multline*}
		\mathcal{A}_{pN}(x)\left[ \Omega_{pN,0} \right] = \begin{bmatrix}
			0_{(p-r)\times r} & \left[ t_{N} \right]_{(p-r)\times (p-r)} \\ \\
			0_{r} & 0_{r \times (p-r)}
		\end{bmatrix} x^N 
		+
		\begin{bNiceMatrix}
			0 & \Cdots & 0 & \ast & \Cdots & \ast \\ 
			\Vdots & p \times r & \Vdots & \Vdots & p \times (p-r) & \Vdots \\ 
			0 & \Cdots & 0 &  \ast & \Cdots & \ast 
		\end{bNiceMatrix}x^{N-1}+O(x^{N-2}).
	\end{multline*}
Noe that 	$\mathcal{A}_{pN-p}(x)$ is as a matrix polynomial of degree $N-1$ whose leading matrix is upper triangular with a nonzero determinant. The product $\mathcal{A}_{pN-p}(x)\left[ \Omega_{pN-p,0} \right]$ will be another matrix polynomial whose leading matrix maintains the number of lower subdiagonals. The elements of the $(p-r)$-th subdiagonal are nonzero, and below this, they will all be zero. The sum of both matrices, $\mathcal{A}_{pN-p}(x)\left[ \Omega_{pN-p,0} \right]$ and $\mathcal{A}_{pN}(x)\left[ \Omega_{pN,0} \right]$, and therefore the first block of $A(x)\Omega$ will maintain the structure presented in the conditions \eqref{LeadingMatrixConditions}, $R(x)$ satisfies the conditions \eqref{CondicionesMatricesLideresFinales}, so we can apply Proposition \ref{prop1} and determine that $\widetilde{\widetilde{\mathcal{A\,}}}_0(x)$ is a constant matrix with non-zero diagonal elements.
	
	On the other hand,   from the orthogonality relations for $A(x)$:
		\begin{align*}
		\begin{aligned}
			\sum_{a=1}^p  \int  \d\mu_{b,a}(x) A_n^{(a)}(x) x^m  &= 0, & m &\in \left\{0,\cdots,\left\lceil\frac{n-b+2}{q}\right\rceil-1\right\},
		\end{aligned}
	\end{align*}
we find for $\in \left\{0,\cdots,\left\lceil\frac{n-b+2}{q}\right\rceil-1\right\}$ that
	\begin{multline*}
	\sum_{a=1}^p  \int \d\mu_{b,a}(x)\left(A_n^{(a)}(x)+ A_{n+1}^{(a)}(x)\Omega_{n+1,n}+ \cdots+A_{n+M}^{(a)}(x)\Omega_{n+M,n}\right)x^m   \\
	= \sum_{a=1}^p  \int \d\mu_{b,a}(x)\widetilde{A}_n^{(a)}(x)x^m =0.
\end{multline*}
Now, noticing that
	\begin{align*}
		\sum_{a=1}^p \int \d\mu_{b,a}(x) \widetilde{A}_n^{(a)}(x) & = \sum_{a=1}^p  \int \left(\d\mu \cdot R\right)_{b,a}(x) \widetilde{\widetilde{A}\,}\hspace*{-1.8pt}^{(a)}_n(x).
	\end{align*}
	we 	 directly deduce the orthogonality relations for $\widetilde{\widetilde{A\,}}(x)$

	Hence,  $\widetilde{\widetilde{A\,}}(x)$ satisfy orthogonality relations:
\[	\begin{aligned}
				\sum_{a=1}^p\int \left(\d\mu\cdot R\right)_{b,a}(x)\widetilde{\widetilde{A}\,}\hspace*{-1.8pt}^{(a)}_n (x)  x^m &= 0,&
				m \in \left\{0, \cdots , \deg{B}_{n-1}^{(b)}\right\},
	\end{aligned}\]
	as well as the  initial condition for the perturbed polynomials 
		\begin{equation*}
			\widetilde{\widetilde{\mathcal A\,}}_0=	\begin{bNiceMatrix}
					\nu_0^{(1)} & \nu_1^{(1)} & \Cdots & \nu_p^{(1)}  \\[3pt]
					0 & \nu_1^{(2)} & \Cdots &\nu_p^{(2)} \\
					\Vdots &  \Ddots& \Ddots[shorten-end=-5pt] & \Vdots \\
					0 & \Cdots &  0& \nu_p^{(p)}
				\end{bNiceMatrix}.
			\end{equation*}
These implies that $\widetilde{\widetilde{A}\,}\hspace*{-1.8pt}^{(a)}_n$ is a family of perturbed mixed multiple orthogonal polynomials.

Next, we proceed to demonstrate the result in the necessity; i.e. that the condition $\tau_n = 0$ for some $n\in\N$ implies the absence of perturbed orthogonality. Beginning with Equation \eqref{KernerlConexMP} for the $b$-th entry, we apply the various eigenvectors and evaluate the expressions at their respective eigenvalues to obtain:
\begin{align*}
	-\begin{bNiceMatrix}
		\mathbb{K}_b^{[n]}(x_1,y) \\[3pt]
		\mathbb{K}_b^{[n]}(x_2,y) \\
		\Vdots \\
		\mathbb{K}_b^{[n]}(x_M,y)
	\end{bNiceMatrix} = \begin{bNiceMatrix}
		\mathbb{A}_n(x_1) & \mathbb{A}_{n+1}(x_1) & \Cdots & \mathbb{A}_{n+M-1}(x_1) \\[3pt]
		\mathbb{A}_n(x_2) & \mathbb{A}_{n+1}(x_2) & \Cdots & \mathbb{A}_{n+M-1}(x_2) \\
		\Vdots & \Vdots & & \Vdots \\
		\mathbb{A}_n(x_M) & \mathbb{A}_{n+1}(x_M) & \Cdots & \mathbb{A}_{n+M-1}(x_M) \\
	\end{bNiceMatrix}\begin{bmatrix}
		\Omega_{M}
	\end{bmatrix} \begin{bmatrix}
		\hat{B}(y)
	\end{bmatrix}_b,
\end{align*}
It is noted that we have recovered the notation introduced in the proof of Equation \eqref{ConexDefinitivaBMP}. Given that $\tau_n = 0$ has been assumed, there exists a nonzero  vector $\boldsymbol{c}=	\begin{bNiceMatrix}
	c_1 & c_2 & \Cdots & c_M
\end{bNiceMatrix}$ such that:
\begin{equation*}
	\begin{bNiceMatrix}
		c_1 & c_2 & \Cdots & c_M
	\end{bNiceMatrix}\begin{bNiceMatrix}
		\mathbb{A}_n(x_1) & \mathbb{A}_{n+1}(x_1) & \Cdots & \mathbb{A}_{n+M-1}(x_1) \\[3pt]
		\mathbb{A}_n(x_2) & \mathbb{A}_{n+1}(x_2) & \Cdots & \mathbb{A}_{n+M-1}(x_2) \\
		\Vdots & \Vdots & & \Vdots \\
		\mathbb{A}_n(x_M) & \mathbb{A}_{n+1}(x_M) & \Cdots & \mathbb{A}_{n+M-1}(x_M) \\
	\end{bNiceMatrix} = \begin{bNiceMatrix}
		0 & 0 & \Cdots & 0 
	\end{bNiceMatrix}
\end{equation*}
Regarding the CD kernel polynomials (multiplied and evaluated by the corresponding eigenvectors and eigenvalues), this condition implies that:
\begin{align*}
	\begin{bNiceMatrix}
		c_1 & c_2 & \Cdots & c_M
	\end{bNiceMatrix}\begin{bNiceMatrix}
		\mathbb{K}_b^{[n]}(x_1,y) \\[3pt]
		\mathbb{K}_b^{[n]}(x_2,y) \\
		\Vdots \\
		\mathbb{K}_b^{[n]}(x_M,y)
	\end{bNiceMatrix} =0, 
\end{align*}
that is
\[	c_1\mathbb{K}_b^{[n]}(x_1,y) + c_2\mathbb{K}_b^{[n]}(x_2,y)+\cdots+c_M\mathbb{K}_b^{[n]}(x_M,y) = 0, \]
that expands as follows 
\begin{align*}
		c_1\boldsymbol{v}_1\begin{bNiceMatrix}
		K_{1,b}^{[n]}(x_1,y) \\[3pt]
		K_{2,b}^{[n]}(x_1,y) \\
		\Vdots \\
		K_{p,b}^{[n]}(x_1,y) \\
	\end{bNiceMatrix}
	+c_2\boldsymbol{v}_2\begin{bNiceMatrix}
		K_{1,b}^{[n]}(x_2,y) \\[3pt]
		K_{2,b}^{[n]}(x_2,y) \\
		\Vdots \\
		K_{p,b}^{[n]}(x_2,y) \\
	\end{bNiceMatrix}+\cdots+c_M\boldsymbol{v}_M\begin{bNiceMatrix}
		K_{1,b}^{[n]}(x_M,y) \\[3pt]
		K_{2,b}^{[n]}(x_M,y) \\
		\Vdots \\
		K_{p,b}^{[n]}(x_M,y) \\
	\end{bNiceMatrix} = 0
\end{align*} 
and, consequently, we deduce that
\begin{align*}
	\sum_{i=0}^{n-1} \left( c_1\boldsymbol{v}_1\begin{bNiceMatrix}
		A_i^{(1)}(x_1) \\[3pt]
		A_i^{(2)}(x_1) \\
		\Vdots \\
		A_i^{(p)}(x_1) \\
	\end{bNiceMatrix}+c_2\boldsymbol{v}_2\begin{bNiceMatrix}
		A_i^{(1)}(x_2) \\[3pt]
		A_i^{(2)}(x_2) \\
		\Vdots \\
		A_i^{(p)}(x_2) \\
	\end{bNiceMatrix} + \cdots + c_M\boldsymbol{v}_M\begin{bNiceMatrix}
		A_i^{(1)}(x_M) \\[3pt]
		A_i^{(2)}(x_M) \\
		\Vdots \\
		A_i^{(p)}(x_M) \\
	\end{bNiceMatrix} \right) B_i^{(b)}(y) = 0
\end{align*}
By using the linear independence of the polynomials $B_i^{(b)}(y)$, we arrive at: 
\begin{equation*}
	c_1\boldsymbol{v}_1\begin{bNiceMatrix}
		A_i^{(1)}(x_1) \\[3pt]
		A_i^{(2)}(x_1) \\
		\Vdots \\
		A_i^{(p)}(x_1) \\
	\end{bNiceMatrix}+c_2\boldsymbol{v}_2\begin{bNiceMatrix}
		A_i^{(1)}(x_2) \\[3pt]
		A_i^{(2)}(x_2) \\
		\Vdots \\
		A_i^{(p)}(x_2) \\
	\end{bNiceMatrix} + \cdots + c_M\boldsymbol{v}_M\begin{bNiceMatrix}
		A_i^{(1)}(x_M) \\[3pt]
		A_i^{(2)}(x_M) \\
		\Vdots \\
		A_i^{(p)}(x_M) \\
	\end{bNiceMatrix} = 0, \quad i\in \{0,\cdots,n-1\}.
\end{equation*}
Upon examining the relationships for $i\in\{0,\cdots,M-1\}$, we have:
\begin{equation*}
	\begin{bNiceMatrix}
		c_1 & c_2 & \Cdots & c_M
	\end{bNiceMatrix}
	\begin{bNiceMatrix}
		\mathbb{A}_0(x_1) & \mathbb{A}_{1}(x_1) & \Cdots & \mathbb{A}_{M-1}(x_1) \\[3pt]
		\mathbb{A}_0(x_2) & \mathbb{A}_{1}(x_2) & \Cdots & \mathbb{A}_{M-1}(x_2) \\
		\Vdots & \Vdots & & \Vdots \\
		\mathbb{A}_0(x_M) & \mathbb{A}_{1}(x_M) & \Cdots & \mathbb{A}_{M-1}(x_M) \\
	\end{bNiceMatrix} = \begin{bNiceMatrix}
		0 &  \Cdots&  & 0
	\end{bNiceMatrix}.
\end{equation*}
This condition is equivalent to  $\tau_0$ being identically zero. Furthermore, upon examining the leading matrix of the following relation: 
\begin{equation*}
	\mathcal{A}_0(x)\left[\Omega_{0,0}\right]+\mathcal{A}_p(x)\left[\Omega_{p,0}\right] + \cdots + \mathcal{A}_{pN-p}(x)\left[\Omega_{pN-p,0}\right]+\mathcal{A}_{pN}(x)\left[\Omega_{pN,0}\right] = R(x) \widetilde{\widetilde{\mathcal{A}}}_0(x),
\end{equation*}
we observe that
\begin{equation*}
	\left[\Omega_{pN-p,0}\right] = \begin{bNiceMatrix}
		\Omega_{pN-p,0} & \Cdots & & & & \Omega_{pN-p,p-1} \\
		\Vdots & & & & & \Vdots \\
		\dfrac{\tau_0}{\tau_1} \\
		0 & \Ddots \\
		\Vdots&\Ddots \\
		0 & \Cdots & 0 & \dfrac{\tau_{r-1}}{\tau_{r}} & \Cdots & \Omega_{pN-1,p-1}
	\end{bNiceMatrix} 
\end{equation*}
However, since $\tau_0$ is not identically zero, the product $\mathcal{A}_{pN-p}(x)\left[\Omega_{pN-p,0}\right]$ does not satisfy the conditions \eqref{LeadingMatrixConditions} for the sub-leading matrix, and the initial conditions matrix will have a determinant equal to zero, see Corollary \ref{Corollary}.
	\end{proof}
	The preceding proof assumes that all zeros of the determinant are simple. With the introduction provided in Section \S \ref{S:ArbitraryMultiplicity}, the generalization to less stringent cases is analogous to the one carried out here.

\section*{Conclusions and Outlook}

In this paper, we unveil new explicit Christoffel formulas applicable to a wide spectrum of matrix polynomial perturbations, including a highly generalized leading coefficient with arbitrary rank. This advancement significantly extends our previous findings in \cite{AAGMM} for matrix orthogonal polynomials and in \cite{bfm} for multiple orthogonal polynomials. Additionally, we establish the equivalence between the existence of perturbed orthogonality and the non-vanishing of certain determinants, which we term as $\tau$  determinants.

Currently, our focus lies on investigating analogous results for Geronimus and Geronimus–Uvarov perturbations. Yet, establishing an existence result within this framework poses a challenge, primarily due to the absence, so far, of a divisibility result for Cauchy transforms. Additionally, drawing from \cite{Zhe}, we will delve into examining the behavior of the Stieltjes matrix transform under these Christoffel transformations within the realm of mixed multiple orthogonal polynomials. An enticing avenue for future exploration entails extending these general transformations to families of hypergeometric multiple orthogonal polynomials and understanding the transformation of the recurrence matrix. Finally, the KP-Toda  integrable flows requires further investigation within this Christoffel perturbation scenario.
	
	\section*{Acknowledgments}

	The authors acknowledges research project [PID2021- 122154NB-I00], \emph{Ortogonalidad y Aproximación con Aplicaciones en Machine Learning y Teoría de la Probabilidad}  funded  by
	 \href{https://doi.org/10.13039/501100011033}{MICIU/AEI/10.13039 /501100011033} and by "ERDF A Way of making Europe”.
	
	\section*{Declarations}
	
	\begin{enumerate}
		\item \textbf{Conflict of interest:} The authors declare no conflict of interest.
		\item \textbf{Ethical approval:} Not applicable.
		\item \textbf{Contributions:} All the authors have contribute equally.
	\end{enumerate}

\end{document}